\documentclass{amsart}
\usepackage{graphics,verbatim,color,amsthm}

\usepackage[all]{xy}

\theoremstyle{plain}

\newtheorem{proposition}{Proposition}
  
\newtheorem{cor}{Corollary}
 
\newtheorem{definition}{Definition}

\theoremstyle{remark}

\theoremstyle{definition}

\def\metrictwo#1{\langle\!\langle #1 \rangle\!\rangle}

\def\normtwo#1{\| #1 \|}

\def\R{\mathbb{R}}

\def\S{\mathbb{S}}

\def\l{\lambda}
\def\a{\alpha}
\def\b{\beta}
\def\e{\epsilon}

\def\g{\gamma}

\def\d{\delta}

\def\p{\phi}
\def\t{\theta}

\def\w{\omega}

\def\z{\zeta}

\def\cM{{\mathcal M}}
\def\cE{{\mathcal E}}

\def\p2{\frac{\pi}{2}}

\def\into{\rightarrow}

\def\union{\cup}
\def\dim{{\rm dim}\,}

\def\smallskip{\par\vspace{1mm}}
\def\medskip{\par\vspace{2mm}}
\def\bigskip{\par\vspace{3mm}}

\def\fr#1#2{\frac{#1}{#2}}

\def\m#1{\begin{bmatrix}#1\end{bmatrix}}

\newcount\equationcount
\def\thenumber{0}
\def\eq#1{\global\advance\equationcount by 1
   \def\thenumber{\number\equationcount}
                        {$$#1\eqno(\thenumber)$$}}

\begin{document}

\title[Scattering in the Isosceles 3BP]{Scattering in the positive energy isosceles three-body problem}
\author{Richard Moeckel}

\date{February 19, 2026}

\maketitle

\begin{abstract}
In the three-body problem with positive energy, solutions which avoid triple collision have the property that the size of the triangle formed by the bodies tends to infinity as $t\into \pm\infty$.  Furthermore, the triangles have well-defined asymptotic shapes $s_\pm$.  The scattering problems asks which asymptotic shape $s_+$ can occur for a given choice of $s_-$.  Previous work shows that this can be viewed as the problem of finding heteroclinic orbits connecting equilibrium points on a boundary manifold ``at infinity'' and some results were obtained for solutions which avoid collisions.  The goal of this paper is to study the scattering effect of binary and near-triple collisions in a simple setting -- the isosceles three-body problem.  The details depend on the mass parameters but in many cases, a fixed isosceles initial shape $s_-$ scatters to essentially all possible isosceles shapes $s_+$.
\end{abstract}

\section{Introduction}
The Newtonian $3$-body problem models the motion of three point masses $m_i$ with positions $q_i\in \R^d$, $i=1,2,3$.  The configuration vector $q=(q_1,q_2,q_3)\in\R^{3d}$ describes the triangle formed by the three bodies.  If the center of mass is fixed at the origin then the moment of inertia $r$ where $r^2=\normtwo{q}^2= m_1|q_1|^2+m_2|q_2|^2 +m_3|q_3|^2$ measures the size of the triangle formed by the three points.  
Two kinds of collision singularities are possible.  A binary collision occurs at time $T$ if exactly two of the masses approach the same position as $t\into T$.  A triple collision occurs if all three masses approach the origin, that is, if the size of the triangle converges to zero.  Binary collisions can be regularized.  If there are no triple collisions, then the regularized solutions exist for all time $t\in\R$.

If the total energy is positive and if no triple collisions occur then the regularized solutions have $r(t)\into\infty$ as $t\into\pm\infty$ \cite{Chazy, DMMY}.  In fact, the
size of the triangle $r(t)$ decreases monotically from $\infty$ to some minimal value $r_{min}$ and then increases monotonically to infinity.
Moreover, the configuration has well-defined limiting shapes in forward and backward time where the shape is described by the normalized configuration vector $s= q/r$.  Two shapes  $s_\pm$ are {\em related by scattering} if there is a regularized solution such that $s(t)\into s_-$ as $t\into-\infty$ and $s(t)\into s_+$ as $t\into +\infty$.

In \cite{DMMY} the idea of {\em scattering at infinity} was introduced by considering limits of solutions such whose minimal size $r_{min}$ is large.  Equivalently, at least two of the three sides of the triangle remain large.  Limits of such solutions were studied by introducing an invariant {\em infinity manifold} $\cM_\infty$ as a boundary to the physical energy manifold.  Orbits  with $r_{min}$ large are perturbations of orbits in $\cM_\infty$.  The construction is similar to McGehee's blow-up of triple collision to obtain an invariant {\em triple collision manifold} forming a boundary at $r=0$ \cite{McGehee}. In $\cM_\infty$ there is a normally repelling manifold of  equilibrium points $\cE_-$ and a normally attracting manifold of equilibrium points $\cE_+$.   These manifolds are parametrized by the set of noncollision shapes.  A shape $s_-$ is related by scattering to a shape $s_+$ if and only if there is a heteroclinic orbit connection the corresponding equilibrium points $p_-\in\cE_-$ and $p_+\in\cE_+$.

It turns out that the flow on the infinity manifold is particularly simple.   For noncollision shapes with large size, the effect of gravity disappears and the motion is that of free particles. 
A free particle motion is of the form $q(t) = a+t v$ where $a,v\in\R^{3d}$.  For such a motion, we have $s(t) = q(t)/r(t)\into \pm\fr{v}{\normtwo{v}}$.  In other words, the limiting shapes are just plus or minus the normalized velocity vector.  Since this can be chosen arbitrarily it follows that for every shape $s$ with $\normtwo{s}=1$, $s$ and its inversion $-s$ are related by scattering at infinity.  There are infinitely many heteroclinic orbits connecting the corresponding restpoints, given by fixing the vector $a$.  In \cite{DMMY} binary collisions were not regularized, so strictly speaking we should choose $a$ such that the free particle motion avoids collisions.  This can always be done and then the free particle motion really does give an orbit in $\cM_\infty$.  Of course we are mainly interested in scattering orbits in the physical energy manifold rather than in $\cM_\infty$.

The main results of \cite{DMMY} are as follows.  First, the flow near the normally hyperbolic manifolds of noncollision equilibrium points $\cE_\pm$ are analytically linearizeable.  The equilibrium points have analytic stable and unstable manifolds and a scattering orbit will be an intersection of $W^u(p_-)$ and $W^s(p_+)$ for $p_\pm \in \cE_\pm$.  If we fix the restpoint $p_-$ the final restpoint $p_+$ will depend analytically on the initial condition in $W^u(p_-)$ provided $p_+\in \cE_+$, that is, provided it's a restpoint corresponding to a noncollision shape.  Next, the free particle orbits in $\cM_\infty$ are shadowed by orbits with $r(t)<\infty$ and $r_{min}$ large.  Finally, the scattering for orbit near infinity is nontrivial in the sense that a given shape $s_-$ is related by scattering to a nonempty open set of shapes $s_+$ where the heteroclinic orbits have no collisions and have large $r_{min}$.

The goal of this paper is to investigate the scattering effect of binary and triple collisions in a simple case -- the isosceles three-body problem.  Suppose two masses are equal say, $m_1=m_2=1$.  Then there is an invariant subset of the planar three-body problem such that the third mass $m_3$ remains on the $y$-axis while $q_1$ and $q_2$ are related by reflection through this axis.  The possible isosceles shape can be parametrized by a single variable $\t\in [-\p2,\p2]$ (see Figure~\ref{fig_isoscelesshapes}).  The endpoints $\t=\pm\p2$ are binary collision shapes and $\t=0$ represents a flat triangle with the bodies collinear. 
\begin{figure}[h]
\scalebox{0.5}{\includegraphics{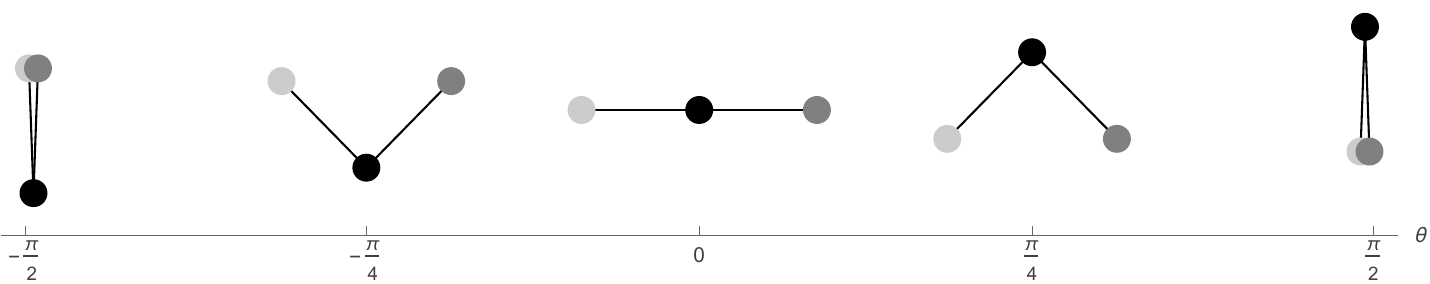}}
\caption{Isosceles shapes.}
\label{fig_isoscelesshapes}
\end{figure}

To study scattering in the isosceles problem, imagine fixing some initial shape $\t_-$.  We will see below that for a given positive energy $h$, there is a one-parameter family of orbits asymptotic to this shape as $t\into -\infty$.  Figure~\ref{fig_scatterm31thetalag} show what happens to this family as $t$ increases toward $+\infty$.  The family is plotted in the $(\t,v)$-plane where $v$ is a velocity variable asymptotic to $\dot r$, the rate of change of the size.  For the energy level $h=\fr12$ it turns out that for solutions with noncollision asymptotic shapes, $v\into \pm1$ as $t\into\pm\infty$.  In the figure, the solutions start near $v=-1$ with $r$ large and and equilateral triangular shape.  They are followed until they return near $v=1$ with $r$ large again.  It seems that the possible final shapes $\t_+$ cover the entire shape interval $[-\p2,\p2]$.  
\begin{figure}[h]
\scalebox{0.4}{\includegraphics{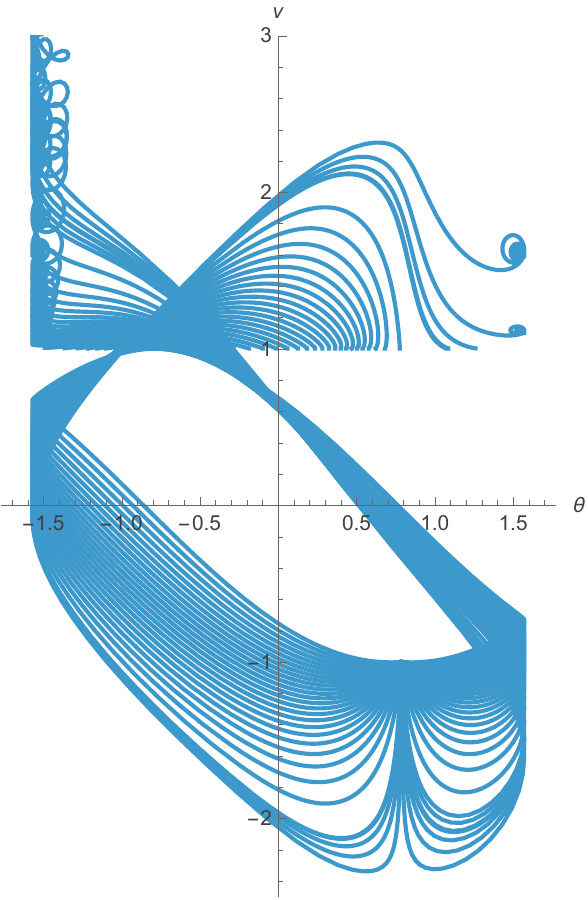}}
\caption{Scattering from the equilateral triangle shape in the equal mass isosceles three-body problem.  }
\label{fig_scatterm31thetalag}
\end{figure}

\begin{figure}[h]
\scalebox{0.4}{\includegraphics{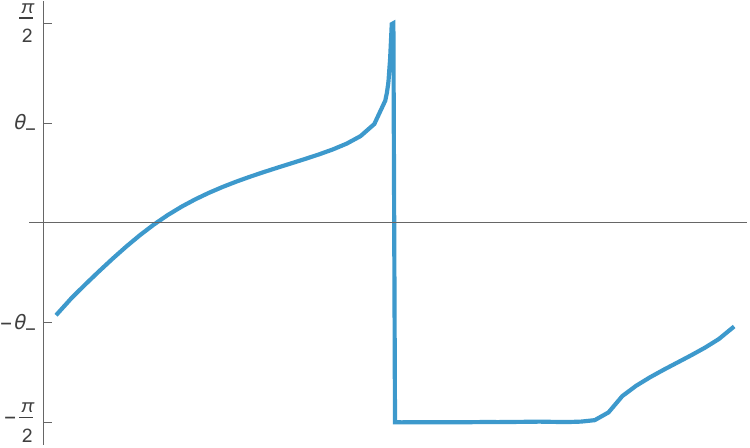}}
\caption{One parameter family of final shapes $\t_+$ with $\t_-$ fixed at the equilateral triangle as in Figure~\ref{fig_scatterm31thetalag}.  The discontinuity occurs at a triple collision orbit. In this case $\t_+$ varies over $[-\p2,\p2]\setminus\{-\t_-\}$.}
\label{fig_thetaplusgraphm31}
\end{figure}

This is further illustrated in Figure~\ref{fig_thetaplusgraphm31} which graphs the final shape $\t_+$ for the one-parameter family of orbits starting at $\t_-$.  Similar results hold for any initial shape $\t_-$, not just the equilateral one.  It will be proved in Proposition~\ref{prop_scatteringcaseII} below that for this equal mass case, the possible values of $\t_+$ do indeed cover the whole shape interval with the possible exception of $-\t_-$.  

Actually, the value $\t_+=-\t_-$ is also attained if we allow scattering at infinity.   Recall that studying scattering at infinity amounts to considering free particle motion, as long as that motion does not involve collisions.  But for the isosceles problem collisions are inevitable, even for free particles.  In section~\ref{sec_regularization} below, we will regularize the binary collisions and then in section~\ref{sec_flowatinfinity} we study the resulting flow on the isosceles infinity manifold $\cM_\infty$.   It turns out that every shape $\t_-$ is related to the shape $-\t_-$ by scattering at infinity.  For example the shape with $\t=\fr{\pi}{4}$ in Figure~\ref{fig_isoscelesshapes} is related to the shape with $\t=-\fr{\pi}{4}$.  

To see the effect of the binary collision it's interesting to compare this to the result stated above for free particle motions with no collisions.  Let $s = (q_1,q_2,q_3)$ be the normalized isosceles shape represented by $\t=\fr{\pi}{4}$.  Then there are free particle motions in the plane with no binary collisions connection the shape $s$ to the shape $-s$.  The latter shape is just the shape $s$ rotated by $\pi$.  While superficially similar, this is not the same as the shape represented by $\t=-\fr{\pi}{4}$  (see Figure~\ref{fig_binaryeffect}).

\begin{figure}[h]
\scalebox{0.4}{\includegraphics{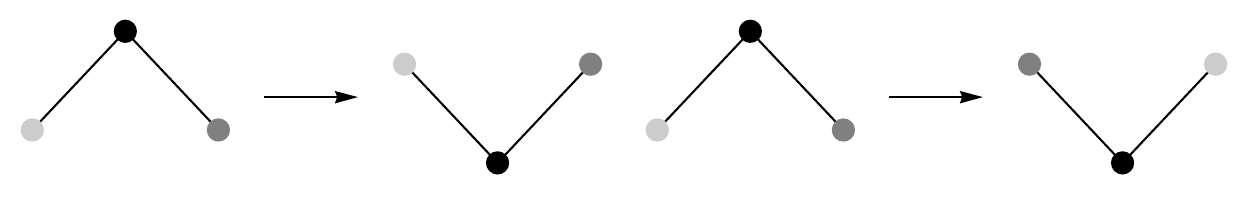}}
\caption{Scattering at infinity with and without binary collisions.  On the left we have the isosceles problem with binary collision.  On the right, a possible free particle scattering in the plane without collisions.  The final shapes are different and even not rotationally equivalent in the plane.}
\label{fig_binaryeffect}
\end{figure}

Triple collision also has a large effect on scattering.  After McGehee's pioneering study of triple collision in the collinear three-body problem there were several studies of the isosceles problem \cite{Simo, Devaney, SimoMartinez}.    The invariant triple collision manifold $\cM_0$ is a two-dimensional surface.  Solutions which experience triple collision converge to $\cM_0$, that is, $r\into 0$.  Since they never return near $r=\infty$, the final scattering angle $\t_+$ is undefined.  The discontinuity in Figure~\ref{fig_thetaplusgraphm31} occurs at a parameter value associated to a triple collision solution.  The key point for us is the behavior of $\t_+$ for solution near triple collision.  This will be determined by the flow on $\cM_0$.

The dynamics on $\cM_0$ depends on the choice of the mass parameters. Simo showed numerically that there are two bifurcation values $m_3=\e_1,\e_2$ where the dynamics qualitatively changes.  We have  $\e_1\approx 0.378532$ and $\e_2\approx 2.661993$.  Propositions ~\ref{prop_scatteringcaseI}, \ref{prop_scatteringcaseII} and \ref{prop_scatteringcaseIII} below present our main scattering results for each of the three mass intervals $(0,\e_1), (\e_1,\e_2)$ and $(\e_2,\infty)$.  For the equal mass case $m_3=1$ we have $\e_1<m_3<\e_2$.  For this interval, it turns out that for solutions near triple collision have $\t_+=\pm\p2$.  In Figure~\ref{fig_thetaplusgraphm31}, consider the interval to the left of the discontinuity.   Its endpoints represent orbits near $\cM_\infty$ and $\cM_0$.  Since $\t_+$ varies continuously, it takes on all values between $-\t_0$ and $\p2$.  On the other side of the discontinuity the endpoints are near  $\cM_0$ and $\cM_\infty$ and we get scattering angles between $-\p2$ and $-\t_0$.  Thus we can get any $\t_+$ in $[-\p2,\p2]$ except that the value $\t_+=-\t_-$ occurs only at infinity.  

The detailed results  are slightly different for $m_3$  in each of the parameter intervals.  See Propositions ~\ref{prop_scatteringcaseI}, \ref{prop_scatteringcaseII} and \ref{prop_scatteringcaseIII}.  Roughly speaking for $m_3$ in $(\e_1,\e_2)$ and $(\e_2,\infty)$ we get scattering to essentially all possible angles $\t_+$ in the interval $[-\p2,\p2]$ while for $m_3$ in $(0, \e_1)$ we can only guarantee about half of the interval.  But in all cases  it's the difference between the behaviors near infinity and near triple collision which is responsible for the large range of $\t_+$.  

The paper is divided into many sections.  In Section~\ref{sec_Jacobi} we blow up the triple collision and introduce the isosceles shape parameter $\t$.  Section~\ref{sec_regularization} introduces another shape parameter $u$ which is convenient for regularizing binary collisions.  Section~\ref{sec_triple} recalls the previous work on isosceles triple collision.  Section~\ref{sec_Chazy} is devoted to showing how some of the classical work of Chazy \cite{Chazy} applies to the regularized isosceles problem.  In particular, we need to compare the classical time variable $t$ to the time variable $\tau$ introduced in previous sections.  Sections~\ref{sec_infinitymanifold} and \ref{sec_flowatinfinity} describe the isosceles infinity manifold $\cM_\infty$ and study scattering at infinity, then in section~\ref{sec_collisiontoinfinity} we find connections from $\cM_\infty$ to triple collision.  Our main scattering theorems are proved in section~\ref{sec_scatteringtheorems} and finally, these are illustrated with numerical examples in section~\ref{sec_scatteringexperiments}.

\section{Blown-up Jacobi Coordinates}\label{sec_Jacobi}
Consider the isosceles three-body problem with masses $m_1=m_2=1$, $m_3>0$.  Let $q_1=(-x_1/2,y),q_2=(x_1/2,y),q_3=(0,y_3)$ be the positions where $x_1\ge 0$ and $2y+m_3 y_3=0$.  Introduce Jacobi coordinates
$$q_2-q_1=(x_1,0)\qquad  q_3-\fr12(q_1+q_2) = (0,y_3-y) = (0,x_2).$$
Let $x=(x_1,x_2)$.  Then the Lagrangian is
$$\begin{aligned}
L(x,\dot x) &= \fr14|\dot x_1|^2+\fr12\mu|\dot x_2|^2+U(x)\\
U(x) &= \fr{1}{d_{12}}+\fr{m_3}{d_{13}} + \fr{m_3}{d_{23}} \\
d_{12}&=|x_1|\quad d_{13}=d_{23} = \sqrt{x_2^2+\fr14 x_1^2}
\end{aligned}
$$
where $\mu= \fr{2m_3}{2+m_3}$.
The Euler-Lagrange equations are:
\begin{equation}\label{eq_odeJacobi}
\fr12 \ddot x_1= U_{x_1}\qquad \mu\,\ddot x_2 = U_{x_2}.
\end{equation}
Solutions of the Euler-Lagrange equations preserve the energy $\fr14|\dot x_1|^2+\fr12\mu|\dot x_2|^2 - U(x) = h$.
It's convenient to define a {\em mass norm} and {\em mass inner product} on $\R^2$ by 
$$\normtwo{v}^2=\fr12 v_1^2+\mu v_2^2\qquad \metrictwo{v,w} =  \fr12 v_1 w_1+\mu v_2 w_2.$$
The energy equation becomes
$$\fr12\normtwo{\dot x}^2-U(x)=h.$$

Let $r^2=\normtwo{x}^2=\fr12|x_1|^2+\mu|x_2|^2$ be the moment of inertia and define blown-up coordinates $(r,\t)$ by
$$x_1=\sqrt{2}r\cos\t\qquad x_2 = \fr{1}{\sqrt\mu}r\sin\t.$$ 
The Lagrangian becomes
$$\begin{aligned}
L(r,\t,\dot r,\dot \t) &= \fr12(\dot r^2+r^2\dot\t^2)+\fr{1}{r}V(\t)\\
V(\t) &= \fr{1}{r_{12}}+\fr{2m_3}{r_{13}} \\
r_{12}&=\sqrt{2}|\cos\t| \quad r_{13}= \sqrt{\fr{\cos^2\t}2+\fr{\sin^2\t}{\mu}}
\end{aligned}
$$
We have $-\fr\pi 2 \le \t \le \fr\pi 2$.  Solutions of the Euler-Lagrange equations preserve the energy
$$\fr12(\dot r^2+r^2\dot\t^2)-\fr{1}{r}V(\t) = h.$$

Introduce a new timescale $\tau$ with
$$\fr{d\tau}{dt} = r \psi(r)$$
where $\psi(r)$ is some smooth function.  Setting $\psi(r) = \sqrt{r}$ gives the usual McGehee timescale factor $r^\fr32$ which is good for studying triple collision at $r=0$.  On the other hand using
$$\psi(r) = \sqrt{\fr{r}{1+r}}$$
gives a timescale factor which behaves like McGehee's near $r=0$ but is asymptotic to $r$ as $r\into\infty$.  The latter seems good for studying scattering..

Define new velocity variables
$$v = \psi(r) \dot r \qquad \a = r\psi(r) \dot\t.$$
Then the blown-up Euler Lagrange differential equations are
$$\begin{aligned}
r^, &= v\,r\\
v^,&= \fr1{2(1+r)}v^2+\a^2-\fr{1}{1+r}V(\t) \\
\t^,&= \a\\
\a^, &= -\fr{1+2r}{2(1+r)}v\a + \fr{1}{1+r}V_\t(\t)
\end{aligned}
$$
where $\;^,$ denotes differentiation with respect to the new timescale.
The energy equation is
$$\fr12(v^2+\a^2) - \fr{1}{1+r}V(\t) = \fr{hr}{1+r}.$$

\section{Regularization of binary collisions}\label{sec_regularization}
The  angular variable $\t$ parametrizes the {\em shape space}, that is, the set of configurations with $\normtwo{x}^2=\fr12|x_1|^2+\mu|x_2|^2=1$.  The normalization defines an ellipse and the shape space is the semi-ellipse with $x_1\ge 0$.  The endpoints $\t = \pm \fr\pi 2$ correspond to binary collision shapes with masses $m_1,m_2$ colliding.  We are going to regularize the double collisions using a two-to-one cover of the shape space by a circle.  This resembles the Levi-Civita approach to regularization in the planar problem.

Let $(y_1,y_2) \in \S^1 = \{y_1^2+y_2^2=1\}$ and define a branched covering map by setting 
$$x_1=\sqrt{2}y_1^2\qquad x_2= \fr{1}{\sqrt{\mu}}y_2\sqrt{1+y_1^2}.$$
This is a map from  $\S^1$ onto the semi-ellipse which is two-to-one except at the poles $(y_1,y_2)=(0,\pm 1)$ which map to the  binary collision points.
If we parametrize the circle by setting $(y_1,y_2)=(\cos u, \sin u)$, $u\in\R$ we have an infinite branched cover from $\R$ to the shape space.  The angular variables $\t,u$ are related by
$$\cos\t=\cos^2u \qquad \sin\t= \sin u\sqrt{1+\cos^2u}$$
or, explicitly
\begin{equation}\label{eq_thetau}
\t(u)=\arctan\fr{ \sin u\sqrt{1+\cos^2u}}{\cos^2u} \qquad \t_u(u) = \fr{2\cos u}{\sqrt{1+\cos^2u}}.
\end{equation}
We have $\t(u+2\pi)=\t(u)$ and $\t(u+\pi)= -\t(u)$.

As $u$ varies over $[-\fr\pi{2},\fr\pi{2}]$, $\t$ also sweeps out the interval $[-\fr\pi{2},\fr\pi{2}]$ and this interval is repeatedly covered as $u$ varies over $\R$.  As $u$ continues over 
$[\fr\pi{2},\fr{3\pi}{2}]$, $\t$ again sweeps out the interval $[-\fr\pi{2},\fr\pi{2}]$ but backwards.  This is the double cover of shape space.  But using the angular variable this oscillating double cover is repeated with period $2\pi$ (see Figure~\ref{fig_thetau}).  While the variable $u$ is useful for regularization, using $\t\in [-\fr\pi{2},\fr\pi{2}]$  gives a unique way to describe the shape.
\begin{figure}[h]
\scalebox{0.5}{\includegraphics{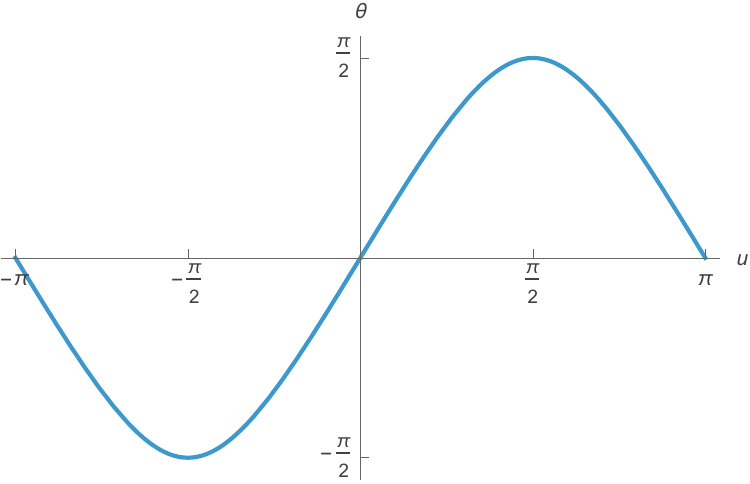}}
\caption{The graph of the function $\t(u)$ relating the two shape variables.  As $u$ varies over an interval of length $2\pi$, the shape parameter $\t$ is double-covered except at the binary collisions.}
\label{fig_thetau}
\end{figure}

To complete the regularization we will also introduce a new velocity variable $w=\a\cos u$ and a new timescale $\tau$ such that 
$$\fr{dt}{d\tau} = \sqrt{\fr{r}{1+r}}r\cos^2u.$$
The extra factor of $\cos^2u$ slows things down near binary collisions.
Using ${}'$ to denote differentiation with respect to the new time scale, the regularized differential equations become
\begin{equation}\label{eq_regode}
\begin{aligned}
r' &= vr\cos^2u\\
v' &= \fr{1}{2(1+r)} v^2\cos^2u +w^2-\fr{1}{1+r}W(u) \\
u' &= \fr12 \sqrt{1+\cos^2u}\,w\\
w' &=-\fr{1+2r}{2(1+r)}vw\cos^2u +\fr{\sqrt{1+\cos^2u}}{2(1+r)}W_u(u) \\
&\qquad\qquad + \fr12\left(v^2-\fr{2hr}{1+r}\right)\sqrt{1+\cos^2u}\sin u\cos u
\end{aligned}
\end{equation}
where 
$$W(u) =\cos^2u\,V(\t(u)) = \fr{1}{\sqrt{2}}+\fr{2m_3}{\sqrt{\cos^4u/2+(1-\cos^4u)/\mu  }}.$$
The energy equation is
\begin{equation}\label{eq_renergy}
\fr12(v^2\cos^2u+w^2) - \fr{1}{1+r}W(u) = \fr{hr}{1+r}\cos^2u.
\end{equation}
The regularized potential function $W(u)$ is smooth near the binary collisions.   Figure~\ref{fig_VW} shows the two shape potential functions $V(u), W(u)$.

\begin{figure}[h]
\scalebox{0.6}{\includegraphics{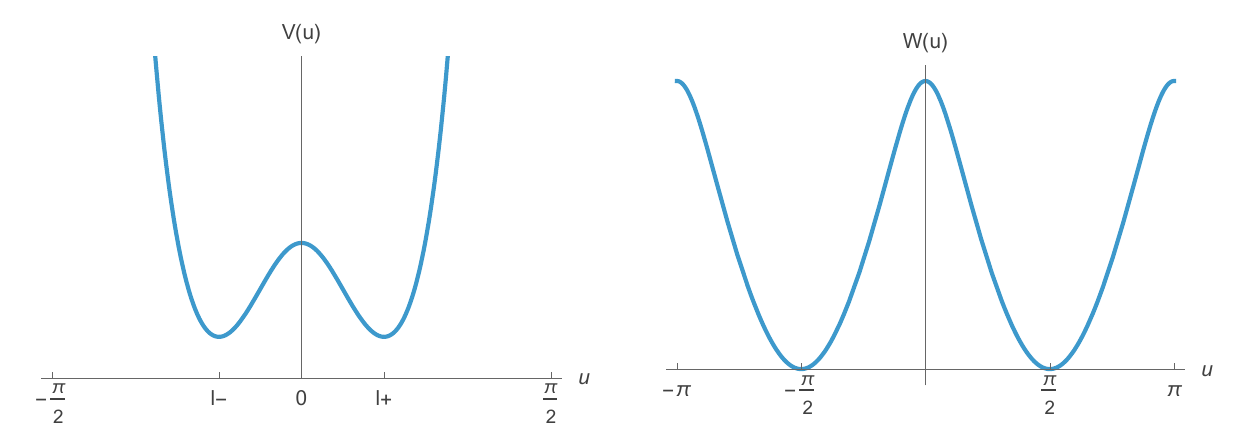}}
\caption{Shape potentials for the equal mass isosceles three-body problem ($m_3=1$).  $V(u)$ (left) has singularities at $\pm \fr{\pi}{2}$.  The regularized potential $W(u)$ (right) is smooth.  $l_-,l_+$ of $V$ are the Lagrange, equilateral critical points.}
\label{fig_VW}
\end{figure}

Finally, a coordinate system which displays both triple collision and infinity is obtained by using the radial variable 
$$s=\fr{r}{1+r} \qquad r = \fr{s}{1-s}.$$
The differential equations are
\begin{equation}\label{eq_sode}
\begin{aligned}
s' &= vs(1-s)\cos^2u\\
v' &= \fr12(1-s)v^2\cos^2u +w^2-(1-s)W(u) \\
u' &= \fr12 \sqrt{1+\cos^2u}\,w\\
w' &=-\fr12(1+s)vw\cos^2u +\fr12\sqrt{1+\cos^2u}(1-s)W_u(u) \\
&\qquad\qquad +\fr12 \left(v^2-2hs \right)\sqrt{1+\cos^2u}\sin u\cos u
\end{aligned}
\end{equation}
with energy equation
\begin{equation}\label{eq_senergy}
\fr12(v^2\cos^2u+w^2) -(1-s)\,W(u) =hs \cos^2u.
\end{equation}
Note that  $\{s=0\}$ and $\{s=1\}$ are both invariant sets.  $s=0$ represents triple collision while $s=1$ corresponds to $r=\infty$. 

The next result shows that all of these coordinates changes yield a global flow on a three-dimensional energy manifold.  The four-dimensional phase space for (\ref{eq_regode}) is
$\R^+ \times \R^3=\{(r,v,u,w):r\ge 0\}$.  For (\ref{eq_sode}) it's $[0,1] \times \R^3=\{(s,v,u,w):0\le s\le 1\}$.  
\begin{proposition}\label{prop_existence}
The energy equations (\ref{eq_renergy}) and (\ref{eq_senergy}) define smooth three-dimensional submanifolds of phase space and the solutions $\g(\tau)$ of 
equations (\ref{eq_regode}) and (\ref{eq_sode}) exist for $\tau\in\R$.
\end{proposition}
\begin{proof}
The energy equation (\ref{eq_senergy}) can be written as $E(s,v,u,w)=\fr12(v^2 \cos^2u+w^2) -(1-s)\,W(u) -hs \cos^2u = 0$ and
$$\nabla E = (W(u)-h\cos^2u,v\cos^2u,E_u,w).
$$
If $w\ne 0$ or $W(u)\ne h\cos^2u$ then $\nabla E\ne 0$.  If $w= 0$ and $W(u)= h\cos^2u$ then the energy equation gives $v^2\cos^2u = 2W(u)$.  Since $W(u)>0$ this means
$v\cos^2u\ne 0$ and again $\nabla E\ne 0$.  So (\ref{eq_senergy}) defines a smooth submanifold.  A similar argument works for (\ref{eq_renergy}).

To see that solutions exist for all time, first note that since $r\ge 0$ and $s\in [0,1]$ and since $W(u)$ is bounded, the energy equations give finite upper bounds for $w^2$.  Since 
$|u'|\le \sqrt{2}|w|$, $u(\tau)$ can't blow up in finite time.
There is no upper bound for  $|v|$ but from (\ref{eq_sode}) and (\ref{eq_senergy}) we have
$$-W(u)\le v' \le v^2\cos^2u+w^2 \le 2W(u)+2h.$$
Since $v'(\tau)$ is bounded above and below, $v(\tau)$ cannot blow up in finite time.  Since $s\in [0,1]$ and since $r'\le vr$ these variables are also bounded for finite times.
\end{proof}
Remember that $u$ is an angular variable representing a point of $\S^1$.  Let $[u] =(\cos u,\sin u)\in\S^1$ and define
$[(r,v,u,w)]=(r,v,[u],w), [(s,v,u,w)]=(s,v,[u],w)$.  The phase spaces can now be viewed as $\R^+ \times \R\times \S^1\times\R$ and $[0,1] \times \R\times \S^1\times\R$.
The extra compactness is occasionally useful below.

In the classical theory of the three-body problem, the only impediments to global existence of solutions are collision singularities.  We have blown-up the triple collision singularity into an invariant subset which will be investigated in the next section.  The solutions of the regularized equations simply pass through the binary collisions.
The binary collision occur at times $\tau_c$ with $\cos u(\tau_c)=0$, that is, at times such that $u(\tau_c)\equiv \fr{\pi}{2} \bmod \pi$.  If $\cos u=0$, the energy equation shows that
$u' =w$ with $w^2=(1-s)W(u)$.    It follows that for binary collisions with $s\ne 1$, the binary collision times $\tau_c$ are isolated.

\section{Triple collision}\label{sec_triple}
 In our blown-up coordinates, triple collision occurs at $\{s=0\}$, which is an invariant manifold for (\ref{eq_sode}).   
 \begin{equation}\label{eq_sode0}
\begin{aligned}
s' &= 0\\
v' &= \fr12v^2\cos^2u +w^2-W(u) \\
u' &= \fr12 \sqrt{1+\cos^2u}\,w\\
w' &=-\fr12vw\cos^2u +\fr12\sqrt{1+\cos^2u}(W_u(u)  +v^2\sin u\cos u)
\end{aligned}
\end{equation}
with energy equation
\begin{equation}\label{eq_senergy0}
\fr12(v^2\cos^2u+w^2) -W(u) =0.
\end{equation}
So we have an invariant two-dimensional {\em collision manifold} (see Figure~\ref{fig_collisionmanifold})
$$\cM_0(h) = \{(s,v,u,w): s=0, (\ref{eq_senergy0})\text{ holds} \}.$$
 \begin{figure}[h]
\scalebox{0.5}{\includegraphics{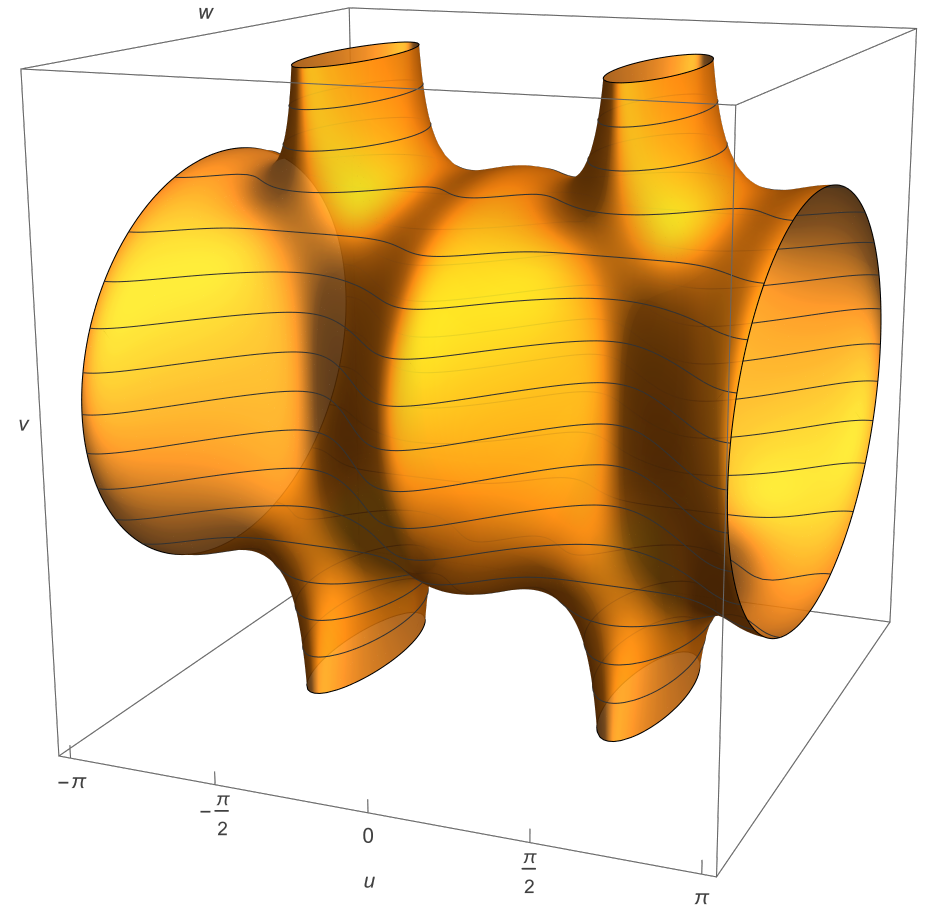}}
\caption{Triple collision manifolds for the isosceles three-body problem  in $(u,w,v)$ coordinates .  Level curves of the Lyapunov function $v$ are shown. The manifold repeats periodically in $u$ with period $2\pi$. }
\label{fig_collisionmanifold}
\end{figure}
 
 The flow near triple collision has been extensively studied \cite{Simo,SimoMartinez,Devaney}. We will review some of the pertinent results here.  We will concentrate on the fundamental region with  $-\fr{\pi}{2}\le u\le\fr{\pi}{2}$.   In Figure~\ref{fig_VW}  three critical points of $V(u)$ for the case $m_3=1$ are visible.  These correspond to the well-known central configurations of the three-body problem.  For the isosceles case, these are the collinear Euler configuration  with $m_3$ at the midpoint of $m_1,m_2$ and two equilateral triangular Lagrange configurations.  The Euler shape is given by $u=0$.  The Lagrange shapes are at $u=l_+, l_-$ where $l_- = -l_+$ and
$$\cos^4 l_+ = \frac{2+m_3}{2(1+2m_3)}.$$

There are  six corresponding restpoints in the fundamental region of the collision manifold which play an important role in this paper.  
\begin{proposition}\label{prop_equilibrium0}
The equilibrium points on $\cM_0(h)$ with  $-\fr{\pi}{2}\le u\le\fr{\pi}{2}$ are the points 
$$(s,v,\t,w) = (0,\pm\sqrt{2V(u_c)},u_c,0)$$
 where $u_c$ is one of the three critical points of the unregularized potential $V(u)$.
\end{proposition}
\begin{proof}
At a restpoint of (\ref{eq_sode0}) we must have $w=0$.  The energy equation reduces to $v^2\cos^2u=2W(u)$ or
$v^2=2V(u)$ where $V(u)$ is the unregularized potential.  This automatically gives $v'=0$.  The equation $w'=0$ becomes
$W_u(u) + v^2\sin u\cos u=0$.  This is not satisfied at the double collision points where $\cos\t=0$ and if $\cos\t\ne 0$, a short computation shows that  it reduces to $V_u=0$.
\end{proof}

Of course these restpoints will be repeated periodically in $u$.  The restpoints on $\cM_0$ are all hyperbolic and their stable and unstable manifolds determine the behavior of solutions of the three-body problem close to triple collision.  Let
$L_\pm=(0,\sqrt{2V(l_\pm)},l_\pm,0)$ denote the restpoints with the Lagrange configuration $l_\pm$ and with $v>0$ and $L^*_\pm=(0,-\sqrt{2V(l_\pm)},l_\pm,0)$ denote those with $v<0$.  Similarly $E,E^*$ will denote the restpoints with the collinear, Euler configuration.    If $p$ is any one of these restpoints we will use the notation 
$W^{s,u}_0(p)$ to denote the stable and unstable manifolds viewed within the two-dimensional collision manifold $\cM_0$.  It turns out that the four Lagrange restpoints are all saddles while $E,E^*$ are a sink and source respectively:
$$\begin{aligned}
&\dim W^s_0(L_\pm)=1 &\dim W_0^u(L_\pm)=1 &\quad \dim W^s_0(L^*_\pm)=1 &\dim W^u_0(L^*_\pm)=1 \\
&\dim W_0^s(E)=2 &\dim W_0^u(E)=0&\quad \dim W_0^s(E^*)=0 &\dim W_0^u(E^*)=2
\end{aligned}
$$
Furthermore eigenvalues at $E,E^*$ are non-real for $m_3<\fr{55}{4}$.

The notation $W^{s,u}(p)$ will denote  the stable and unstable manifolds viewed in the three-dimensional energy manifold.  It turns out that all of the starred restpoints, that is, those with $v<0$, get an extra stable direction while the unstarred restpoints get an extra unstable one.  Thus,
$$\begin{aligned}
&\dim W^s(L_\pm)=2 &\dim W^u(L_\pm)=1 &\quad \dim W^s(L^*_\pm)=1 &\dim W^u(L^*_\pm)=2 \\
&\dim W^s(E)=2 &\dim W^u(E)=1&\quad \dim W^s(E^*)=1&\dim W^u(E^*)=2.
\end{aligned}
$$
These results follow from eigenvalue computations similar to those in the references cited above and will not be repeated here.  While the coordinate systems used are slightly different, this only changes the eigenvalues by a real, positive factor.

The stable manifolds of the starred restpoints are not entirely contained in the collision manifold.  Any initial condition in $W^s(L_{\pm}^*)$ or $W^s(E^*)$ with $s>0$ represents a solution of the three-body problem which has a forward-time triple collision.  Similarly, any solution with $s>0$ in $W^u(L_{\pm})$ or $W^u(E)$ has a backward-time triple collision.
Proposition~\ref{prop_triplecollision} below shows that these stable and unstable manifolds represent all of the triple collision solutions.

We want to characterize all solutions of (\ref{eq_sode}) which experience triple collision.  It turns out that answering questions like this about  
the global dynamics of the problem for energies $h\ge 0$ is relatively simple compared with the negative energy case.  One aspect of this is existence of a Lyapunov function for the flow.  It's the same function used in McGehee's work after taking into account our different choice of time rescaling.
\begin{proposition}\label{prop_lyapunov}
 Suppose the energy satisfies $h\ge 0$.  Then the function $\phi(\tau) = \fr{v(\tau)}{\sqrt{1-s}}$ is strictly increasing along  solutions of  (\ref{eq_sode}) with $0\le s(\tau)<1$ except at the restpoints on the collision manifold.  On the triple collision manifold $\phi(\tau) = v(\tau)$.
\end{proposition}
\begin{proof}
A straightforward computation using (\ref{eq_sode}) and  (\ref{eq_senergy}) yields the formula
$$\sqrt{1-s}\,\phi' = \fr12w^2+hs\cos^2 u \ge 0.$$
The right-hand side vanishes if and only if $w=s\cos u=0$.  The energy equation shows $w=\cos u=0$ is impossible.  If $w=s=0$ and $\phi(\tau)$ is constant then we must have $w'=0$.  As in the proof of Proposition~\ref{prop_equilibrium0} this implies that we are at a restpoint.
\end{proof}

Since $s'=vs(1-s)\cos^2u$, the sign of $v$ determines whether the size variable $s(\tau)$ is increasing or decreasing.  The following result gives a rough classification for global behavior of the solutions of (\ref{eq_sode}).
\begin{proposition}\label{prop_triplecollision}
Let $\g(\tau)$ be any solution of (\ref{eq_sode}) with $0\le s(\tau)<1$.  If $v(\tau)<0$ for all $\tau\in\R$ then $\g(\tau)$ converges to one of the three restpoints on the collision manifold with $v=-\sqrt{2V(u_c)} <0$ as $\tau\into\infty$ (triple collision in forward time).  If $v(\tau)>0$ for all $\tau$ then $\g(\tau)$ converges to one of the  restpoints on the collision manifold with $v=\sqrt{2V(u_c)}>0$ as $\tau\into -\infty$ (triple collision in backward time).   For all other solutions there is a unique time $\tau_0$ such that
such that $v(\tau)<0$ for $\tau<\tau_0$ and $v(\tau)>0$ for $\tau>\tau_0$. For such solutions either $s(\tau)=0$ for all $\tau$ (solutions in the collision manifold) or else $s(\tau)>0$ and attains a positive minimum at $\tau=\tau_0$.
\end{proposition}
\begin{proof}
Consider a nonequilibrium solution with $v(\tau)<0$ for all $\tau\in\R$.  Then $s(\tau)$ is nonincreasing it has a limit $s(\tau)\into s_0\in [0,1)$ as $\tau\into\infty$.  
Since $\phi(\tau) = \fr{v(\tau)}{\sqrt{1-s}}\le 0$ is strictly increasing, it follows that $v(\tau)$ also has a limit, say $v(\tau)\into v_0\le 0$. 
Also, since the binary collision times are isolated, $s(\tau)$ is either identically zero or strictly decreasing.  We have seen that $w(\tau)$ is bounded. 

Recall that $u(\tau)$ is an angular variable representing a point on a circle.  For $u\in\R$ we defined $[u]=(\cos u,\sin u)\in\S^1$ to be the corresponding point on the circle and 
$[\g(\tau)] = (s(\tau),v(\tau),[u(\tau)], w(\tau))$.  It follows from the last paragraph that $[\g(\tau)]$ has a nonempty, compact, connected omega limit set $\omega([\g])$ contained in 
$\{s=s_0,v=v_0\}$.  If $[p]\in\w([\g])$, and $p=(s_0,v_0,u_0,w_0)$ is any angular representative then the Lyapunov function $\phi(\tau) = \fr{v_0}{\sqrt{1-s_0}}$ is constant along the orbit of $p$.  It follows from Proposition~\ref{prop_lyapunov} that  $p$ is one of the restpoints with $v<0$ in $\cM_0(h)$.  Since the restpoints are isolated, $\w([\g])$ consists of just one point of the form $[p]$ and  $\g(\tau)$ converges to one of its represenatives.  The case 
with $v(\tau)>0$ for all $\tau$ is similar.

For all other solutions, there is some time $\tau_0$ such that $v(\tau_0)=0$.  When $v=0$ (\ref{eq_sode}) and (\ref{eq_senergy}) show that
$$v'=(1-s)W(u)+ 2hs\cos^2u >0.$$
It follows that $\{v>0\}$ is positively invariant and that $\{v<0\}$ is negatively invariant.  The proposition follows.
\end{proof}
We will see in the next section that for the solutions which do not experience triple collisions, $s(\tau)\into 1$ and hence $r(\tau)\into \infty$ as $\tau\into\pm \infty$.

For the scattering results below we need to understand what happens to solutions with $s>0$ which pass close to triple collision. For example, a solution starting close to the stable manifold $W^s(L_+^*)$ will approach the collision manifold near $L_+^*$ and then closely follow one of the two branches of $W^u(L_+^*)$.  These branches move in $\cM_0$ in the direction of increasing $v$ and will either converge to one of the unstarred restpoints or else spiral up one of the arms near $u=\pm\fr{\pi}{2}\bmod \pi$.  The behavior depends on the choice of masses.  We will use the following result from the numerical work of  Simo \cite{Simo, SimoMartinez} which is best stated using the shape variable $\t\in [-\fr{\pi}{2},\fr{\pi}{2}]$.
\begin{proposition}[Simo]\label{prop_collisionwus}
There are two critical values of $m_3$, $\e_1<\e_2$ such that there are saddle-saddle connections in $\cM_0$.  For other values of $m_3$ the behavior of the branches of $W^u(L_+)$ is as follows.
For $0<m_3<\e_1$, the branch beginning with $w>0$ spirals up the arm with $\t=-\fr{\pi}{2}$ while the branch beginning with $w<0$ ends at $E$.  For $\e_1<m_3<\e_2$, the branch beginning with $w>0$ spirals up the arm with $\t=-\fr{\pi}{2}$ while the branch beginning with $w<0$ spirals up the arm with  up the arm with $\t=\fr{\pi}{2}$.  For $\e_2<m_3$, the branch beginning with $w>0$ ends at $E$  while the branch beginning with $w<0$ spirals up the arm with  up the arm with $\t=\fr{\pi}{2}$. 
The critical values are approximately $\e_1\approx 0.378532$ and $\e_2\approx 2.661993$.
\end{proposition}
The behavior of the stable and unstable manifolds of the other saddle points in $\cM_0$ can be deduced from reflection symmetry and time reversal.

Figure~\ref{fig_collisionwus} shows these numerically computed branches for each of the three nondegenerate cases.  The middle case $m_3=1$ is also shown in the $(u,v)$ plane in Figure~\ref{fig_collisionwuu} for comparison.  The latter figure also shows  the boundary curves of the projection of $\cM_0$.
 \begin{figure}[h]
\scalebox{0.4}{\includegraphics{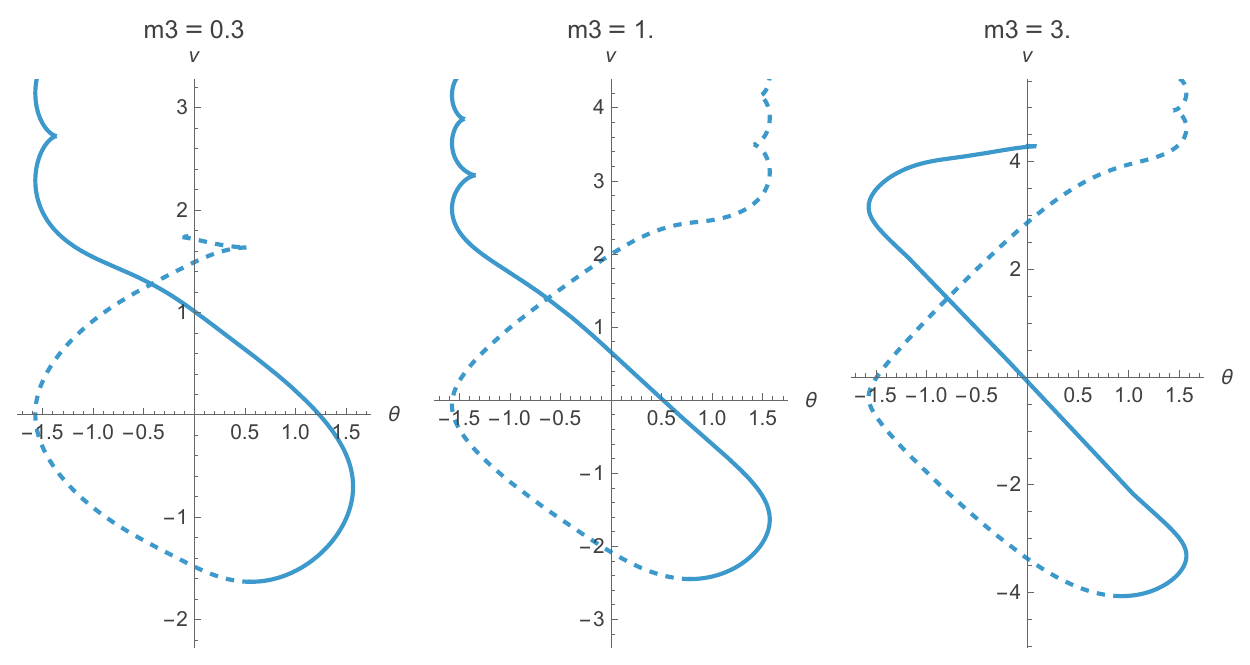}}
\caption{Branches of $W^u(L^*_+)$ projected to the $(\t,v)$-plane.  The three mass values shown illustrate the three cases from Proposition~\ref{prop_collisionwus}.  Solid lines represent the branches beginning with $w>0$ and the dashed represent those beginning with $w<0$.}
\label{fig_collisionwus}
\end{figure}

\begin{figure}[h]
\scalebox{0.4}{\includegraphics{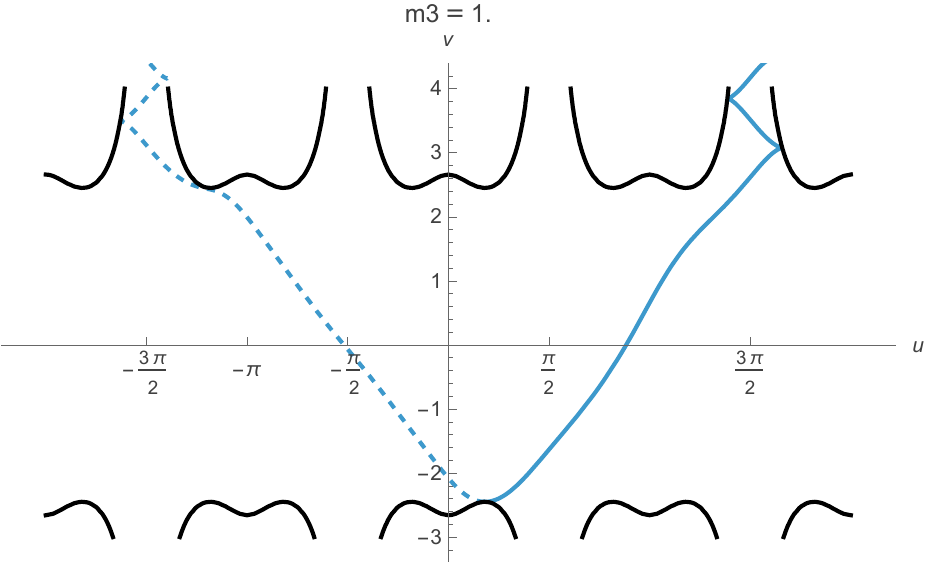}}
\caption{Branches of $W^u(L^*_+)$ projected to the $(u,v)$-plane for $m_3=1$ (corresponding the middle case in Figure~\ref{fig_collisionwus}).  The projection of $\cM_0$ to this plane is the region bounded by the wavy  black curves.}
\label{fig_collisionwuu}
\end{figure}

\section{Estimates in the natural timescale}\label{sec_Chazy}
By replacing the usual time parameter $t$ by the parameter $\tau$ we have regularized the binary collisions and slowed down the triple collisions producing a global flow.  But certain results are easier to prove and understand using the natural time scale.  In this section we temporarily go back to $t$ to prove some classical results due to Sundman and Chazy.
We restrict attention to solutions in the classical phase space $0<r$ or $0<s<1$.

Recall that the factor relating the two  time parameters $t, \tau$ is
\begin{equation}\label{eq_timechange}
t'(\tau) =\frac{r^\fr32}{\sqrt{1+r}}\cos^2u = \frac{s^\fr32}{1-s}\cos^2u
\end{equation}
Proposition~\ref{prop_existence} shows that solutions $\g(\tau)$ of the regularized differential equations exist for all $\tau\in\R$.  Given such a solution we will define a corresponding 
{\em classical time function} by
$$t(\tau) = t_0 +\int_{\tau_0} t'(\tau)\,d\tau$$
where $t_0, \tau_0$ are any convenient choices for initial times.  Note that $t'(\tau)\ge 0$ with equality only when $\cos u=0$, that is, at binary collision points.  
Since the  binary collision times $\tau_c$ are isolated $t(\tau)$ is still strictly increasing in some neighborhood of the collision time.  It follows that the binary collision times are isolated in both timescales.  

In fact, consider a collision with $u(\tau_c)=\fr{\pi}{2}$.  We have a convergent series expansion $u(\tau)= \fr{\pi}{2} + u'(\tau_c)(\tau-\tau_c)+\ldots$ where 
$u'(\tau_c) = \fr12w(\tau_c)\ne 0$.  Then $\cos^2u(\tau) = u'(\tau_c)^2(\tau-\tau_c)^2+\ldots$ and 
 integration of (\ref{eq_timechange}) gives a convergent series expansion of the form
\begin{equation}\label{eq_tseries}
t-t_c =  \frac{s_c^\fr32}{1-s_c}\fr{u'(\tau_c)^2}{3}(\tau-\tau_c)^3+\ldots
\end{equation}
where $s_c=s(\tau_c)$.  

Let $\g(\tau)$ be any solution of $(\ref{eq_regode})$ with $0<r(\tau)$ and let $t(\tau)$ be the corresponding classical time function.  The image $J(\g) = t(\R)$ will be some open interval and there will be a continuous,  inverse function $\tau(t)$, $\tau:J(\g)\into \R$.  the components of $\g(\tau(t)) =  (r(t),v(t),u(t),w(t))$ will be continuous functions of $t$ for $t\in J$.  The Jacobi coordinates 
$$x_1(t) = \sqrt{2}r(t)\cos^2u(t)\qquad x_2(t) = \fr{1}{\sqrt{\mu}}r(t) \sin u(t)\sqrt{1+\cos^2u(t)}$$
will also extend continuously through the binary collisions.  In fact, by (\ref{eq_tseries}), the inverse function $\tau(t)$ has a convergent series expansion in powers of $(t-t_c)^\fr13$ of the form
\begin{equation}\label{eq_tauseries}
\tau-\tau_c =  k(t-t_c)^\fr13+\ldots\qquad k\ne 0.
\end{equation}
We will call $x(t)=(x_1(t),x_2(t))$, $t\in J(\g)$ the {\em regularized classsical solution} corresponding to the solution $\g$.

The classical approach to the three-body problem makes use of the well-known Lagrange-Jacobi identity.  If $x(t)$ is a solution of (\ref{eq_odeJacobi}) then
$$\ddot I(t) = 4 h +2U(x(t))$$
where $I(t)=\normtwo{x(t)}^2 = r(t)^2$ is the moment of inertia.  In the positive energy case it follows that $\dot I(t)$
is strictly increasing along solutions.  A priori,  these equations only apply away from collisions but the following proposition gives an extension to regularized classical solutions.
\begin{proposition}\label{prop_Iestimates}
Let $x(t)$ be a regularized classical solution corresponding to a solution $\g$ of (\ref{eq_regode}).  Then $r(t)$ and the moment of inertia
$$I(t) = \fr12 x_1(t)^2+\mu x_2(t)^2 = r(t)^2$$
are $C^1$ for $t\in J(\g)$.  Moreover
\begin{equation}\label{eq_Iestimates}
\begin{aligned}
\dot I(t) &\ge \dot I(0)+4h t\\
I(t)&\ge I(0)+ \dot I(0) t+ 2ht^2
\end{aligned}
\end{equation}
\end{proposition}
\begin{proof}
Assume without loss of generality that the initial times are chosen as $t_0=\tau_0=0$.  We know that $I(t)$ extends continuously to the binary collision times.  Away from binary collisions we have $\dot I = x_1\dot x_1+ 2\mu x_2\dot x_2$ and, from (\ref{eq_odeJacobi}),
$$\ddot I = {\dot x_1}^2+ 2\mu {\dot x_2}^2 + 2 x_1 U_{x_1} + 2 x_2 U_{x_2} = 4 T(\dot x)-2U(x) = 4h + 2U(x).$$
We have
$$U(x(t)) = \fr{1}{x_1(t)}+\fr{2m_3}{\sqrt{x_1(t)^2/2+x_2(t)^2/\mu}.}$$
Locally near a binary collision, $x_1(t)=\sqrt{2}r(t)\cos^2u(t)$ has a series expansion in powers of $(t-t_c)^\fr13$ whose lowest order term is of order $(t-t_c)^\fr23$ while $x_2(t_c)\ne 0$.  It follows that $U(x(t))$ is an integrable function and that
$$\dot I(t) = \dot I(0) + 4ht +\int_0^t U(x(\a))\,d\a$$
gives a continuous extension of $\dot I(t)$ for $t\in J(\g)$.  Moreover, since $U(x)>0$,  $\dot I(t)\ge \dot I(0)+4ht$.  Another integration of this continuous extension gives the estimate for  $I(t)$.  Since $r(t) = \sqrt{I(t)}$ and $I(t)>0$, it follows that $r(t)$ is also $C^1$.

A direct proof that $r(t)$ is $C^1$ using (\ref{eq_regode}) is also possible. Namely, in the regularized timescale $r(\tau)$ is analytic. Since $r' = vr\cos^2u$ it follows that at the binary collision times we have $r'(\tau_c) = r''(\tau_c)=0$.   Then we have a convergent series expansion of the form $r(\tau) = r(\tau_c)+ \fr16 r'''(\tau_c)(\tau-\tau_c)^3 +O((\tau-\tau_c)^4)$.   Substituting the series (\ref{eq_tauseries}) for $\tau - \tau_c$ gives a series for $r(t)$ in powers of $(t-t_c)^\fr13$ whose lowest term is of order $t-t_c$.  This shows that $r(t)$ is locally $C^1$.
\end{proof}

The following result about the classical time interval $J(\g)$ will allow us to apply some classical results of Chazy to our regularized problem.
\begin{proposition}\label{prop_texistence}
If $\g(\tau)$ experiences triple collision, the collision happens at a finite value of $t$ for the corresponding regularized classical solution and the interval $J(\g)$  has one finite endpoint.  All other regularized solutions exist for all $t\in\R$.
\end{proposition}
\begin{proof}
It suffices to consider the behavior for $\tau\ge 0$.   For a triple collision solution in forward time, $v(\tau)<0$ and $v(\tau)\into v_0<0$.  Moreover, $\cos u(\tau)\into \cos(u_c)\ne 0$ where $u_c$ is one of the critical points of $V(u)$. 
Since $r'=vr\cos^2u$, it  follows that $r(\tau)\into 0$ exponentially, that is $r(\tau)\le k\exp(-l\tau)$, for some positive constants $k,l$ and $\tau$ sufficiently large.  Therefore $t'(\tau)\into 0$ exponentially and the integral of (\ref{eq_timechange}) over $0\le \tau<\infty$ is finite.  

For a solution with no forward-time  triple collision, $r(\tau)$  has a positive minimum $r_0>0$ and then begins to increase so we may assume $r(\tau)> r_0$ and $v(\tau)>0$ for $\tau>0$.  We will show that $t(\tau)\into\infty$ as $\tau\into\infty$.   The proof divides into two cases: either there is a constant $\a>0$ and  sequence of times $\tau_n\into\infty$ such that
$\cos^2u(\tau_n)\ge \a$ or else $\cos u(\tau)\into 0$ as $\tau\into \infty$.  Consider the former case.  As long as $\cos^2 u\ge \a/2$ we will have 
$$t'(\tau) =\sqrt{\fr{r}{1+r}}\,r\cos^2u\ge \sqrt{\fr{r_0}{1+r_0}}\,\fr{r_0\a}{2}.$$
As noted above, the energy equation gives a uniform bound for   $|u'| = |w|$.  It follows that there is a uniform $\d>0$ such that $\cos^2u(\tau)\ge \a/2$ for $|\tau-\tau_n|\le \d$ and during each such time interval $t(\tau)$ must increase by at least $\frac{\d\,r_0^\fr32\a}{2\sqrt{1+r_0}}>0$ so $t(\tau)\into\infty$ as $\tau\into\infty$.

Turning to the case that $\cos u(\tau)\into 0$ as $\tau\into\infty$, assume for the sake of contradiction that $t(\tau)\into T< \infty$ as $\tau\into\infty$.  We will show below that $r(\tau)$ and $v(\tau)$ are bounded.  Then consider the solution $[\g(\tau)]$ obtained from $\g$ by replacing $u(\tau)$ by the corresponding point $[u(\tau)]\in \S^1$.  To get the contradiction note that if $r\le K, v\le K$  for some constant $K$ then since $w(\tau)$ is also bounded, $[\g(\tau)]$ will have a nonempty, compact $\w$-limit set.  If $[p]\in\w([\g])$, the Lyapunov function $\phi = \sqrt{1+r}\,v$ must be constant on the orbit of $p$.  But this means $p$ would be restpoint with $r=0$.  Since $r\ge r_0>0$ this is a contradiction.

In remains to show that the assumption that $\cos u(\tau)\into 0$ and $t(\tau)\into T<\infty$ forces $r, v$ to remain bounded.  We may assume that $\cos^2u\le \fr12$ for all $\tau\ge 0$.  Recall that in Jacobi variables
$$I=r^2 = \fr12 x_1^2 + \mu x_2^2  = r^2\cos^2\t+r^2\sin^2\t.$$
It follows that 
$$\fr12 r^2 \le \mu x_2^2 \le r^2.$$
Since   $r_{13}^2 = x_2^2+\fr14 x_1^2$ it will also be bounded above and below by suitable multiples of $r^2$.   

The Euler-Lagrange differential equation for $x_2(t)$ is
$$\mu\,\ddot x_2 = -\fr{2m_3 x_2}{(x_1^2/4+x_2^2)^\fr32}.$$
The right-hand side is continuous at binary collision times so $x_2(t)$ extends to a $C^2$ function for $t\in [0,T)$.  Moreover, by the estimates in the last paragraph
$|\ddot x_2| \le C/r^2 \le C/r_0^2$ for some constant $C>0$.   It follows by integration  that there is are upper bounds 
$$|x_2(t)|\le L\qquad |\dot x_2(t)|\le L\qquad t\in [0,T).$$
Since $r^2\le  2x_2^2$, $r(t)$ is bounded.

To get a bound for $v(t)$, note that 
\begin{equation}\label{eq_v}
\fr12 \dot I = r\dot r = \fr{r r'}{t'(\tau)} = r\sqrt{\fr{1+r}{r}}v\ge r_0 v
\end{equation}
so it suffices to show $\dot I(t)$ is bounded on $[0,T)$.  We have
$$\dot I = x_1\dot x_1 + 2\mu x_2 \dot x_2.$$
We already know that the second term is bounded.  In spite of the binary collisions, the first term is also bounded.  From the Jacobi energy equation we have
$$x_1^2 \dot x_1^2 \le 4 x_1^2(U(x) + h).$$
Note that $x_1 U(x)$ smooth at the binary collision points.  Then the bounds $0<r_0\le r(t) \le K$ for $t\in [0,T)$ show that $(x_1\dot x_1)^2$ is also bounded, as required.
\end{proof}

Using Propositions~\ref{prop_Iestimates} and \ref{prop_texistence} we get the following corollary about convergence of the size variables.
\begin{cor}\label{cor_restimate}
If $\g(\tau)$ is any solution of (\ref{eq_regode}) which does not experience  forward-time triple collision then $r(\tau)\into\infty$ as $\tau\into\infty$.  For non-triple-collision solutions of (\ref{eq_sode}), $s(\tau)\into 1$ as $\tau\into\infty$.  Moreover, given any constant $K<\sqrt{2h}$ we have the estimates
$$r(\tau) \ge Kt(\tau)\qquad 1- s(\tau)\le \fr{1}{1+Kt(\tau)} $$
for all sufficiently large $\tau$.  A similar result holds in backward time.
\end{cor}
\begin{proof}
Proposition~\ref{prop_texistence} shows that $t(\tau)\into \infty$ as $\tau\into\infty$ and (\ref{eq_Iestimates}) shows that $r(\tau) \ge \sqrt{2h}\,t(\tau)(1+O(1/t(\tau))$.  By definition
$s=\fr{r}{1+r}$ which gives the estimate for $s(\tau)$.
\end{proof}

In the rest of this section, we will apply some results of Chazy \cite{Chazy} to our regularized isosceles problems.  Chapter II of Chazy's paper contains several interesting theorems about the positive energy three-body problem.  Following Sundman's work, he allows regularization of binary collisions which lead to local series solutions in powers of $(t-t_c)^\fr13$ and then describes the asymptotic behavior of regularized solutions which exist for all positive time.  We have shown that the regularization used here is of the type considered by Chazy and Proposition~\ref{prop_texistence} shows that, except for the triple collisions,  all of our regularized classical solutions exist for all $t\in\R$.

The following proposition summarizes  the main the results in Chazy, Chapter II, section 4.
\begin{proposition}[Chazy]\label{prop_lambda}
Let $x(t)$ be a regularized classical solution which exists for all $t\ge 0$ and let $r_{min}$ and $r_{max}$ denote the smallest and largest of the three mutual distances.  Then there is a constant $\l\in [0,1]$ such that 
$$r_{min}(t)/r_{max}(t)\into \l$$
as $t\into\infty$.  Moreover, $\l$ depends continuously on initial conditions.  A similar result holds in backward time.
\end{proposition}

As an important corollary, we get convergence of the shape variables for regularized solutions without triple collisions
\begin{cor}\label{cor_shapeconvergence}
Let $\g(\tau)$ be a regularized solution which does not have a forward-time triple collision.  Then the shape variables $u(\tau)$  converges to a limit $u_+(\g)$ as $\tau\into\infty$ and $u_+(\g)$ depends continuously on initial conditions.  If $\g(\tau)$ has no backward-time triple collision then there is a limit $u_-(\g)$ as $\tau\into-\infty$.  
\end{cor}
\begin{proof}
A given ratio $\l = r_{min}/r_{max}\in [0,1]$ determines at most four isosceles triangles.  $\l=0$ corresponds to the two binary collision shapes. $\l=1$ represents the two Lagrange central configuration shapes.  $\l=\fr12$ could represent the Euler central configuration with $r_{13}=r_{23} = \fr12 r_{12}$ or else two acute triangles with $r_{13}=r_{23} = 2r_{12}$.  Any other ratio determines two acute and two obtuse triangles.  It follows that the values $u\in \R$ corresponding to a given ratio $\l$ is a discrete subset of $\R$.
Hence convergence of the ratio implies convergence of $u$.
\end{proof}

In Chapter II,  section 5,  Chazy discusses the case $\l>0$, which means that the limiting shape is not  a binary collision.  Since $r\into\infty$ this implies that all of the mutual distances tend to infinity.  The next proposition summarizes the implications of Chazy's results for the isosceles problem.
\begin{proposition}[Chazy]\label{prop_hyperbolic}
Suppose $\l>0$.  Then there are constants $A_1>0$ and $A_2\ne 0$, depending continuously on initial conditions,  such that the Jacobi velocities satisfy
$$\dot x_1= A_1+O(1/t) \qquad \dot x_2 =  A_2+O(1/t)$$
as $\tau\into\infty$.  The Jacobi variables themselves satisfy
$$x_1= A_1t+O(\log t) \qquad x_2 =  A_2t+O(\log t).$$
\end{proposition}

From this we can get convergence of the rest of the regularized variables.
 \begin{cor}\label{cor_hyperboliclimits}
For every solution of (\ref{eq_regode}) or (\ref{eq_sode}) with no forward-time triple collision and with $\l>0$, we have 
$$v(\tau) \into \sqrt{2h}\quad w(\tau)\into 0$$
as $\tau\into\infty$.   The solution is called {\em forward-hyperbolic}.  The set of initial conditions leading to a forward-hyperbolic orbit is open.  
\end{cor}
\begin{proof}
In the hyperbolic case, all of the mutual distances tend to infinity and so the potential energy satisfies $U(x)\into 0$.  The Jacobi energy equation then shows that the limiting velocity vector
$A=(A_1,A_2)$ satisfies $\normtwo{A}=\sqrt{2h}$.  Also $r(t)^2 = \normtwo{x(t)}^2 = \normtwo{A}^2 t^2+ O(t\log t)$ so
$$r(t) = \normtwo{A}t+O(\log t).$$ 

From (\ref{eq_v}) and Proposition~\ref{prop_hyperbolic} we have
\begin{equation}\label{eq_v2}
v= \sqrt{\fr{r}{1+r}} \dot r =  \sqrt{\fr{r}{1+r}}\fr{\metrictwo{x,\dot x}}{r} =\sqrt{\fr{r}{1+r}}\fr{\normtwo{A}^2t+\ldots}{\normtwo{A}t+\ldots}
\end{equation}
so $v(\tau)\into \normtwo{A} = \sqrt{2h}$.
Using this,  the energy equation (\ref{eq_senergy}) shows that  $w^2 = 2(1-s)W(u) + \cos^2u(2hs - v^2)$ which converges to zero.

Openness of the set of forward-hyperbolic initial conditions follows from the continuity of $\l$ and the fact that the set of initial conditions leading to triple collision is closed.
\end{proof}

Next consider the case $\l=0$.  We have $\cos u(\tau)\into 0$ and $\t\into \pm\fr{\pi}{2}$ which means that the shape of the configuration is a tight binary with $m_1, m_2$ relatively close and $m_3$ far away.   Although the limiting shape is a binary collision, $d_{12} = |x_1|$ does not converge to $0$.
The energy of the $m_1,m_2$ binary
\begin{equation}\label{eq_binaryenergy}
H_{12} =  \fr14\dot x_1^2 -\fr{1}{x_1}
\end{equation}
plays an important role.

The following summarizes some of the relevant results in Chazy, Chapter II, sections 7-9.
\begin{proposition}[Chazy]\label{prop_lambda0}
For regularized classical solutions with no forward-time triple collision and with $\l=0$, there is a constant $A_2\ne 0$, depending continuously on initial conditions,  such that 
$$\dot x_2 =  A_2+O(1/t)$$
as $\tau\into\infty$.  The Jacobi variables  satisfy
$$x_1= O(t^\fr23) \qquad x_2 =  A_2t+O(\log t).$$
The binary energy (\ref{eq_binaryenergy}) converges to a limit  $h_{12} = h- \mu A_2^2 \le 0$.
\end{proposition}
The solution will called {\em forward-time hyperbolic-parabolic} or {\em hyperbolic-elliptic} according to whether $h_{12}=0$ or $h_{12}<0$.  In the parabolic case, 
$x_1(\tau)\into\infty$ as $\tau\into\infty$ while in the elliptic case, $x_1$ is bounded.

The implications of Proposition~\ref{prop_lambda0} for the blown-up variables are as follows:
 \begin{cor}\label{cor_lambda0limits}
For every solution of (\ref{eq_regode}) or (\ref{eq_sode}) with no forward-time triple collision and with $\l=0$, we have 
$$v(\tau) \into \sqrt{2(h-h_{12})}\quad w(\tau)\into 0$$
as $\tau\into\infty$.    The set of initial conditions leading to a forward-hyperbolic elliptic orbits is open.  
\end{cor}
\begin{proof}
The distance $r_{13}$ tend to infinity and so $|U(x)-\fr{1}{x_1}|\into 0$.  Then Jacobi energy equation gives
$h = h_{12} +\fr12 \mu A_2^2$ so $\mu A_2^2= 2(h-h_{12})$

We have $r(t)^2 = \normtwo{x(t)}^2 = \mu A_2^2 t^2+O(t^\fr43)$ so
$$r(t) = \sqrt{\mu}A_2t+O(t^\fr23).$$ 
From (\ref{eq_v2}) and Proposition~\ref{prop_lambda0} we have
$$v= \sqrt{\fr{r}{1+r}}\fr{\metrictwo{x,\dot x}}{r} =\sqrt{\fr{r}{1+r}}\fr{\mu A_2^2 t+\ldots}{ \sqrt{\mu}A_2t++\ldots}$$
so $v(\tau)\into \normtwo{A} = \sqrt{\mu}A_2 = \sqrt{2(h-h_{12})}$.
Using this,  the energy equation (\ref{eq_senergy}) shows that  $w^2 = 2(1-s)W(u) + \cos^2u(2hs - v^2)$.  This time we have $\cos u\into0$ so again $w^2\into 0$.

Openness of the set of forward-hyperbolic-elliptic initial conditions follows from the continuity of $\l$ and the fact that the set of initial conditions leading to triple collision is closed.
\end{proof}

\section{The Infinity Manifold and the scattering relation}\label{sec_infinitymanifold}
According to Corollary~\ref{cor_restimate}, solutions $\g(\tau)$ of (\ref{eq_sode}) without triple collisions  have $s(\tau)\into 1$.
The invariant set $\cM_\infty=\{s=1\}$ will be called the {\em infinity manifold}.
Setting $s=1$ in (\ref{eq_sode}) gives 
\begin{equation}\label{eq_infinityode}
\begin{aligned}
s' &= 0\\
v' &=w^2\\
u' &= \fr12 \sqrt{1+\cos^2u}\,w\\
w' &=-vw\cos^2u  +\fr12 \left(v^2-2h \right)\sqrt{1+\cos^2u}\sin u\cos u
\end{aligned}
\end{equation}
with energy equation
\begin{equation}\label{eq_infinityenergy}
\fr12(v^2\cos^2u+w^2)=h \cos^2u.
\end{equation}
Note that the Newtonian potential has disappeared.  In this limit of infinitely large configurations, the interparticle forces have vanished.  But the  binary collisions are still having an implicit effect due to the use of regularized variables.

This time the energy equation determines a two-dimensional variety with singularities (see Figure~\ref{fig_infinitymanifold}).  The part with $u \ne \pm\fr{\pi}{2}\bmod \pi$ is a smooth surface.  The lines at $u= \pm\fr{\pi}{2}, w=0$ are contained in the variety and are singular points for the surface and restpoints of (\ref{eq_infinityode}).  

\begin{figure}[h]
\scalebox{0.5}{\includegraphics{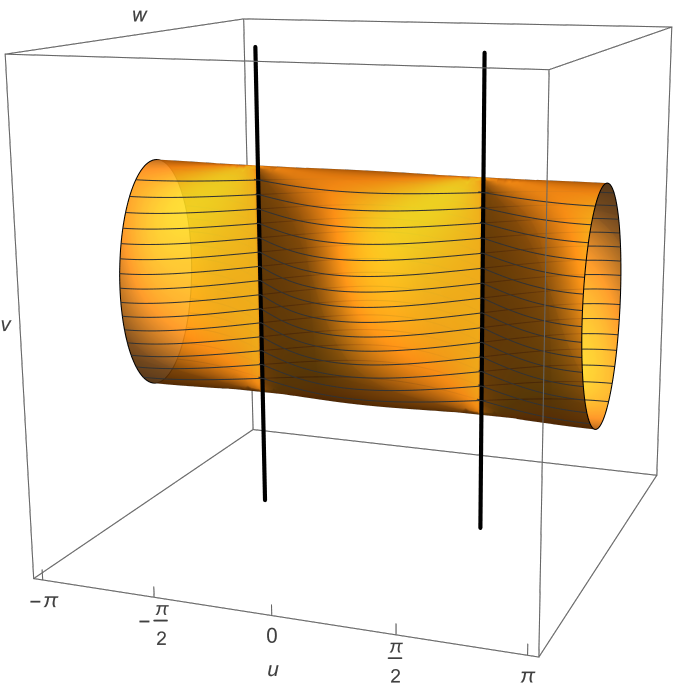}}
\caption{Infinity ``manifold'' $\cM_\infty$ for the isosceles three-body problem with $m_3=1$.}
\label{fig_infinitymanifold}
\end{figure}

The entire energy manifold can be visualized as the solid region in $(u,w,v)$-space trapped between the triple collision manifold and the infinity manifold.  Figure~\ref{fig_bothmanifolds} shows  the projection of the solution set of (\ref{eq_infinityenergy}).  The outer surface is the collision manifold with $s=0$ and the inner one is the infinity manifold.  One can show that, at least for energy $h=\fr12$ the surfaces with values of $s\in [0,1]$ are nested.  

\begin{figure}[h]
\scalebox{0.5}{\includegraphics{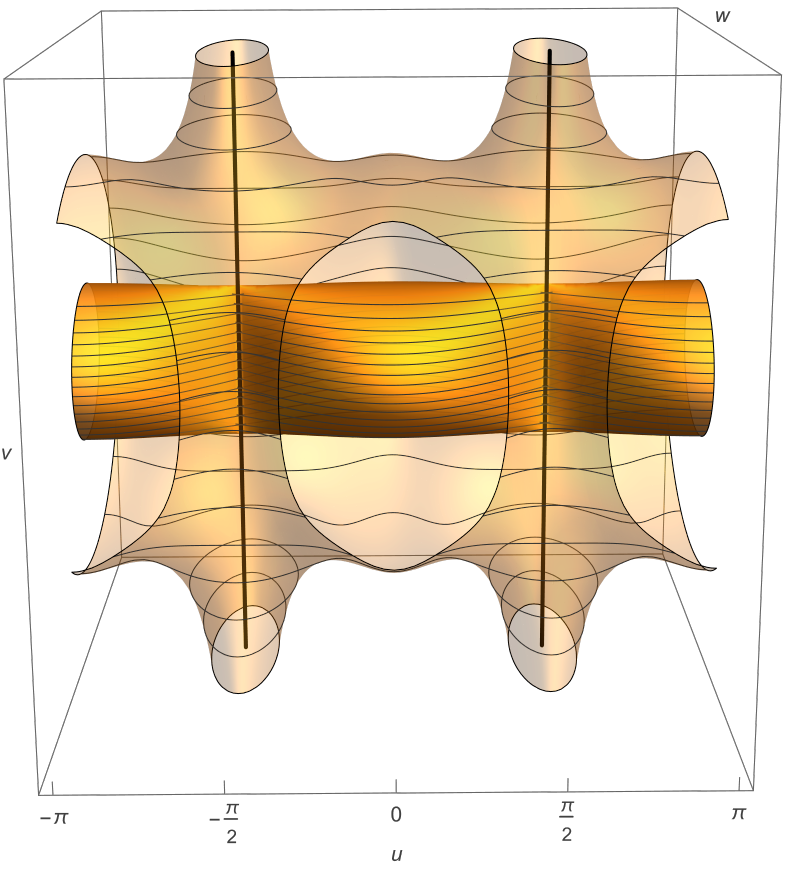}}
\caption{Infinity manifold inside triple collision manifold for $m_3=1$ and $h=\fr12$.}
\label{fig_bothmanifolds}
\end{figure}

Consider the differential equation (\ref{eq_infinityode}) on the infinity manifold.  The vertical lines $\t=\pm\fr{\pi}{2}, w=0$ consist entirely of equilibrium points.   These points are limits of phase space points with $r\into\infty$ and normalized shapes tending to the binary collision shapes.  This would include actual binary collisions points with large $r$ but also points where the Jacobi coordinates have $x_1>0$ but $x_1/r\into \infty$.   Our blown-up coordinates and timescale make all of these ``binary collisions at infinity" equilibrium points but for actual binary collisions with $0\le s<1$ we still have $u' =w\ne 0$.  The horizontal lines $v=\pm \sqrt{2h}, w=0$ on the surface connecting the two vertical ones are also made up of equilibria.  

Corollaries~\ref{cor_shapeconvergence}, \ref{cor_hyperboliclimits} and \ref{cor_lambda0limits} give the asymptotic behavior of solutions without triple collisions.
\begin{proposition}\label{prop_asymptoticbehavior}
Let $\g(\tau)$ be any solution of (\ref{eq_sode}) which does not have a forward-time triple collision.  Then as $\tau\into\infty$, $\g(\tau)$ converges to one of the restpoints on the infinity manifold with $v\ge \sqrt{2h}$.  More precisely, if $\g(\tau) = (s(\tau),v(\tau),\t(\tau),w(\tau))$ we have $s(\tau)\into 1$, $w\into 0$ and
$$ 
\begin{array}{ccc}
v(\tau)\into \sqrt{2h}  &u(\tau)\into u_+(\g)\in (-\fr{\pi}{2},\fr{\pi}{2})\bmod \pi & \text{if }\l(\g)>0\\
v(\tau)\into \sqrt{2(h-h_{12})} &u(\tau)\into u_+(\g)\equiv\pm\fr{\pi}{2}\bmod \pi & \text{if }\l(\g)=0.
\end{array}
$$
The limiting restpoint depends continuously on initial conditions.  Similarly, if $\g(\tau)$ has no triple collision in backward time, it converges as $\tau\into-\infty$ to restpoint with $\t=\t_-(\g)$ and with $v = -\sqrt{2h}$ or $v=-\sqrt{2(h-h_{12})}$ depending on the backward time limit $\l$.
\end{proposition}

Since $h_{12}\le 0$, the possible forward-time limits are the restpoints in Figure~\ref{fig_infinitymanifold} which are on or above the horizontal line segments at $v=\sqrt{2h}$.  Hyperbolic orbits, that is, orbits with $\l(\g)>0$ have limits on the open horizontal segment and nonhyperbolic ones have limits on the vertical ``goal posts'' at $u=\pm\fr{\pi}{2}\bmod \pi$ with $v\ge \sqrt{2h}$.   In backward time the limit points are on or below the horizontal segments at $v=-\sqrt{2h}$.

 If $\g(\tau)$ has no triple collisions in either time direction then it is a heteroclinic orbit between two restpoints on the infinity manifold, one with $v\le -\sqrt{2h}$ and one
 with $v\ge\sqrt{2h}$.  This lead us to the {\em scattering problem}, the main subject of this paper.  
 \begin{definition}
 Let $p_\pm$ be two equilibrium points on the infinity manifold.  A heteroclinic orbit $\g(\tau)$ with $\lim_{\tau\into\pm\infty}= p_{\pm}$ will be called a {\em scattering solution} from $p_-$ to $p_+$.  If such a solution exists then $p_\pm$ and also the shapes $u_{\pm}$ will be said to be {\em related by scattering}.  The same terminology will be applied to the corresponding angles $\t_\pm =\t(u_\pm)\in (-\fr{\pi}{2},\fr{\pi}{2})$.
 \end{definition}
Using the angular variable $\t(u)$ instead of $u$ is a way to uniquely specify the shape (recall Figure~\ref{fig_isoscelesshapes}).
The main question of interest here is:  which asymptotic shapes are related by scattering ?

\section{Flow at Infinity}\label{sec_flowatinfinity}
Our blown-up coordinates are particularly well adapted to study hyperbolic scattering, that is scattering from $\t_-$ to $\t_+$ where both shapes  in $(-\fr{\pi}{2}, \fr{\pi}{2})$.  We begin by studying the linearization at the corresponding equilibrium points  on $\cM_\infty$.  For $\t_0\in(-\fr{\pi}{2}, \fr{\pi}{2})$ there will be a unique representative $u_0\in(-\fr{\pi}{2}, \fr{\pi}{2})$ with $\t(u_0)=\t_0$.  Let $p = (s,v,u,w)=(1,\sqrt{2h},u_0,0)$ be the corresponding equilibrium point with $v=\sqrt{2h}$ in the infinity manifold.  Away from binary collision, the differential equations (\ref{eq_sode}) can be simplified using the energy equation:
\begin{equation}\label{eq_sodesimple}
\begin{aligned}
s' &= vs(1-s)\cos^2u\\
v' &= (1-s)\cos^2u(\fr12v^2-V(u)) +w^2 \\
u' &= \fr12 \sqrt{1+\cos^2u}\,w\\
w' &=-\fr12(1+s)vw\cos^2u +\fr12\sqrt{1+\cos^2u}\left((1-s)\cos^2uV_u(u)-w^2\tan u\right)
\end{aligned}
\end{equation}
The matrix of the linearization of these equations at $p$ has a simple structure:
$$A=\m{-\sqrt{2h}\cos\t_0&0&0&0\\
(V(u_0)-h)\cos^2u_0&0&0&0&\\
0&0&0&\sqrt{1+\cos^2u_0}\\
a_{41}&0&0&-\sqrt{2h}\cos^2u_0
}
$$
where $a_{41}=-\fr12V_u(u_0)\cos^2u_0\sqrt{1+\cos^2u_0}$.
The equation of the tangent space to the energy manifold at $p$ is also very simple:
$$(V(u_0)-h)\delta s+\sqrt{2h}\delta v=0.$$
The eigenvalues of the restriction to the tangent space are $0,\a,\a$ where $\a=-\sqrt{2h}\cos\t_0$.
In the basis 
$$e_1=(0,0,1,0)\quad  e_2=(0,0,\sqrt{1+\cos^2u_0},\a)\quad  e_3=(\a,(V(u_0)-h)\cos^2u_0,0,a_{41})$$
 the matrix  of the restriction of $A$ is
$$\m{0&0&0\\
0&\a&-V_\t(\t_0)\\
0&0&\a
}.
$$
The zero eigenvalue and eigenvector are explained by the fact that $p$ is part of a line segment of  restpoints $(1,\sqrt{2h},u,0)$, $u\in (-\fr{\pi}{2}, \fr{\pi}{2})$.   The other two eigenvalues are attracting and equal with a nontrivial Jordan block except at the critical points of $V(u)$.  The attracting eigenvector $e_2$ is tangent to the infinity manifold.  
This suggests that each individual restpoint $p$ is normally attracting with a two-dimensional stable manifold in the three-dimensional energy space.  That this is indeed the case follows from the following result.
\begin{proposition}\label{prop_analyticlinearization}
Let $\cE_+=\{(s,v,u,w)=(1,\sqrt{2h},\t,0), u\in (-\fr{\pi}{2}, \fr{\pi}{2})\}$ be the line segment of hyperbolic equilibria in the fundamental region of the infinity manifold.  The flow in the neighborhood of each equilibrium point $p$ can be analytically linearized.  $\cE_+$ is a one-dimensional normally attracting invariant manifold of restpoints.  It has a local stable manifold $W^s_{loc}(\cE_+)$ which is an open neighborhood of $\cE_+$ in the three-dimensional  energy manifold and which is foliated into the two-dimensional local stable manifolds of the individual restpoints.    Similarly the set of restpoints $\cE_-$ with $v=-\sqrt{2h}$ is a normally repelling invariant manifold with an open  local unstable manifold, foliated into the unstable manifolds of its restpoints. 
\end{proposition}
A more general result for the $n$-body problem was proved in  \cite{DMMY} so no proof will be given here.   Restricting attention to the subset of the energy manifold with $0\le s\le 1$ we find the that local stable and unstable manifolds are again open subsets but now the manifolds of the individual restpoints are two-dimensional half-disks (the parts of the previous disks with $s\le 1$).  The boundaries of these half-disks are in the infinity manifold $s=1$ and consist of the one-dimensional stable or unstable manifold of the restriction of the flow to the infinity manifold.  This is shown schematically in Figure~\ref{fig_hyperbolicinfinity}.   The plane in the figure represents the infinity manifold $s=1$ and the region above it is $\{s<1\}$.  Note that the eigenvector $e_2$ for the attracting eigenvalue  is tangent to the infinity manifold.  It follows that each restpoint has a one-dimensional stable or unstable curve in the infinity manifold tangent to the eigenspace.  As $u_0\into \pm\fr\pi{2}$, the nonzero eigenvalues tend to zero and the eigenvectors become tangent to the line of restpoints.

\begin{figure}[h]
\scalebox{0.5}{\includegraphics{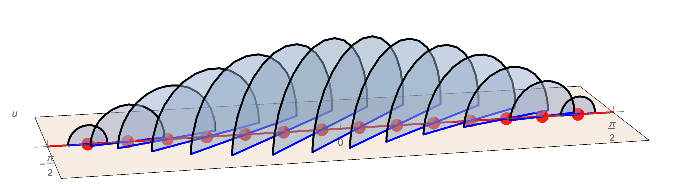}}
\caption{Schematic picture of a neighborhood of hyperbolic infinity in the three-dimensional energy manifold.  The line of restpoints represents $\cE_+$ or $\cE_-$ and the neighborhood is the local stable manifold or unstable manifold, foliated into the manifolds associated to individual restpoints.}
\label{fig_hyperbolicinfinity}
\end{figure}

Because of the analytic linearization,  there is a real analytic projection map from  $W^s_{loc}(\cE_+)$ to $\cE_+$ assigning to each point its limiting equilibrium point.  Recall from Corollary~\ref{cor_shapeconvergence} that for any solution $\g(\tau)$ which does not have a forward-time triple collision, the shape $u(\tau)$ converges to a limit $u_+(\g)$ which depends continuously on initial conditions.   Now we see  that, at least for solutions with $u_+(\g)\ne \pm\fr{\pi}{2}$ the final shape map is actually analytic.

In these terms, the hyperbolic scattering problem starting  at a given $u_-$ amounts to following the unstable manifold of the restpoint $(1,-\sqrt{2h},u_-,0)\in \cE_-$ to a neighborhood of $\cE_+$ and determining which angles $u_+$ occur as final shapes.  Imagine choosing an arc in  $W^u(p)$ transverse to the  flow (bold semicircles in Figure~\ref{fig_hyperbolicinfinity}).   The endpoints of  the arc are in $\cM_\infty$, that is, in the one-dimensional unstable manifolds of $p$ within $\cM_\infty$.  The rest of the arc is in $\{s<1\}$.  In the rest of this section, we will determine what happens to these endpoints and show that the parts of the arc close to the endpoints behaves in a similar way.

Recall that the infinity ``manifold''  $\cM_\infty$ is actually a variety with singularities (see Figure~\ref{fig_infinitymanifold}). Consider the part with $-\sqrt{2h}\le v\le \sqrt{2h}$.  
The energy equation with $s=1$ represents the surface locally as union  of two smooth sheets
\begin{equation}\label{eq_sheets}
w=\pm \cos u \sqrt{2h-v^2}.
\end{equation}
The differential equation (\ref{eq_infinityode}) shows that the function $v$ is a strictly increasing except at the equilibria.  On the half of the surface with $w>0$ the coordinate $u$ increases; on the other half it decreases.  The vertical line segments at $u=\pm\fr{\pi}{2}\bmod \pi$ are also composed of equilibria.
These observations show that there must be many restpoint connections in the surface. 

Imagine starting at a restpoint $p$ on the horizontal segment of $\cE_-$ with $u\in (-\fr{\pi}{2}, \fr{\pi}{2})$ and following the branch $\g$ of the one-dimensional unstable manifold with $w>0$.  $u(\tau)$ will increase while $v(\tau)$ increases and the solution will approach one of the restpoints at $u=\p2$.   We will see below that each such restpoint   has  one-dimensional stable and unstable manifolds in $\cM_\infty$.   These manifolds give a foliation of the part of  $\cM_\infty$ with $-\sqrt{2h}<v<\sqrt{2h}$.  Our solution $\g$ will be in one of these stable curves, say $W^s(q)$ where $q=(1,v,\p2,0)$,  $|v|<\sqrt{2h}$.   Although $\g$ can't pass through the restpoint, we will see below that nearby solutions with $s<1$ will pass through to emerge near $W^u(q)$ with $u>\p2$.   This unstable branch will continue along  to another restpoint and so on.  Define a {\em restpoint chain} in $\cM_\infty$ to be a sequence of such heteroclinic orbits.  Eventually, the chain will terminate at a restpoint in $\cE_+$.

Because the flow on $\cM_\infty$ is independent of the potential, we can have a complete understanding of such restpoint chains.  The behavior of the shape variable $\t$ is particularly simple.
\begin{proposition}\label{prop_totalvariation}  
Consider a chain of restpoints beginning at a restpoint $p_- = (1,-\sqrt{2h},u_0,0)\in\cE_-$ and ending at $p_+=(1,\sqrt{2h},u_1,0)\in\cE_+$.   Then the total variation of $\t(\tau)=\t(u(\tau))$ along  the chain is exactly $\Delta \t=\pi$ and $u_1=u_0\pm\pi$.  There is just one binary collision and $\t_1= -\t_0$, where $\t_i=\t(u_i)$.
\end{proposition}
\begin{proof}
Since $v(\tau)$ is strictly increasing along solutions with $|v|<\sqrt{2h}$ we can parametrize the solutions by $v$.  From (\ref{eq_thetau}) we find that  $\t(v)$ satisfies
$$\fr{d\t}{dv} = \fr{\t_u(u) u'(\tau)}{v'(\tau)} =  \fr{1}{w}.$$
By the energy equation, $w^2=(2h-v^2)\cos \t$ so we get 
$$\left| \fr{d\t}{dv}\right | = \fr{1}{\sqrt{2h-v^2}}\qquad \Delta\t = \int_{-\sqrt{2h}}^{\sqrt{2h}}\fr{dv}{\sqrt{2h-v^2}}  = \pi.$$

Starting at $u_0 \in (-\p2,\p2)$, a chain with $w>0$ will reach $u=\p2$.  Meanwhile, $\t$ has increased from $\t(u_0)$ to $\p2$.  Next $u$ enters $[\p2,\fr{3\pi}{2}]$ and $\t(u)$ begins to decrease.  It must decrease by an additional $\p2+\t_0$ without further binary collision until reaching $-\t_0$.  Since $u=u_0+\pi$ is the first parameter such that  $\t(u) = -\t_0$ we have $u_1=u_0+\pi$.
Similarly, chains beginning with $w<0$ have $\t_0$ decreasing to a binary collision with $\t=-\p2$ and then increasing to $-\t_0$ while $u$ decreases monotonically to $u_0-\pi$.
\end{proof}

Figure~\ref{fig_chains} shows several of these restpoint chains in $\cM_\infty$ mapped to the $(\t,v)$-plane.  We can summarize by saying that the shapes represented by the angles $\t$ and $-\t$ are related by {\em scattering at infinity}.   The paper \cite{DMMY} studied the problem of hyperbolic scattering at infinity for the general $n$-body problem under the assumption that there were no collisions.  In that case, a shape is represented by a normalized configuration of $n$-bodies $q=(q_1,\ldots, q_n)$.  The relevant result there is that the configuration $q$ is related by  scattering  at  infinity to the configuration $-q$.  While superficially similar  to our result  for the isosceles  problem,  this is actually quite different.  For a planar configuration, $q$ and $-q$ are just rotations of one another by $180^\circ$.   But the isosceles configurations with parameters $\t$ and  $-\t$ are reflections of one another through the $x$-axis.  For example the two rotationally inequivalent  equilateral triangles are represented by $\t_{lag}$ and $-\t_{lag}$.   Thus, even though the effect of the potential disappears at infinity, there is still a residual effect due to binary collision.    In his thesis, my student Chen Shi studied the effects of binary collision on the flow at and near infinity in the collinear three-body problem with two equal masses \cite{Shi}.   In that case, although the flow at infinity is still quite simple, it's possible to have restpoint chains with several binary collisions.

\begin{figure}
\scalebox{0.5}{\includegraphics{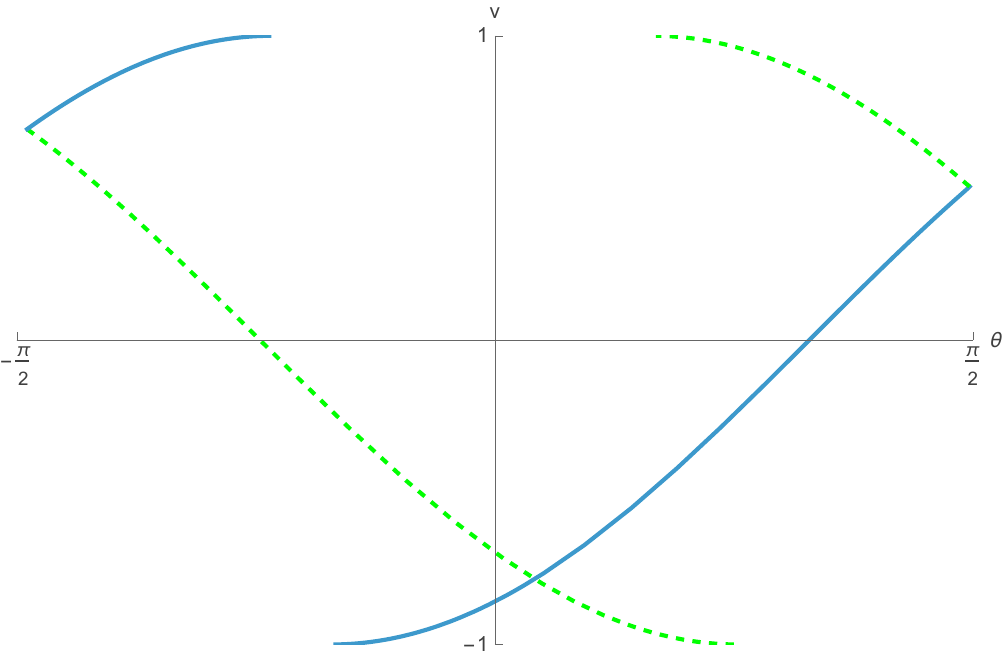}}
\caption{Chains of heteroclinic orbits on the infinity manifold for $h=\fr12$.   The chains connect restpoints 
in $\cE_-$ with a given shape $\t_0$ to restpoints in $\cE_+$ with shape $-\t_0$.}
\label{fig_chains}
\end{figure}

Having understood scattering at infinity,  we want to show that orbits close to infinity behave in a similar way.   The regularized flow is continuous with respect to initial conditions but there is a technical difficulty due  to the fact that the scattering at infinity is represented by a chain of orbits instead  of a single orbit.   For a given restpoint 
$p_- = (1,-\sqrt{2h},\t_0,0)\in\cE_-$ consider a semi-circular arc in $W^u(p_-)$ transverse to the flow as described above.  We will see that the part of the arc close to an endpoint will shadow the corresponding chain of restpoints in the collision manifold and end at a shape $\t_+$ close to $-\t_-$.  

\begin{proposition}\label{prop_shadowing}
Let $p_- = (1,-\sqrt{2h},u_0,0)\in\cE_-$ and let $\a(z)$, $z\in [0,1)$, be a continuous curve with $\a(0)\in W^u(p_-)\cap \cM_\infty$ and $\a(z)\in \{s<1\}$ for $z>0$.   Then  for $\e>0$ sufficiently small, the asymptotic shapes $u_+(\a(z))$ and $\t_+(\a(z))$, $z\in [0,\e)$ are well-defined and continuous.
\end{proposition}

\begin{figure}[h]
\scalebox{0.5}{\includegraphics{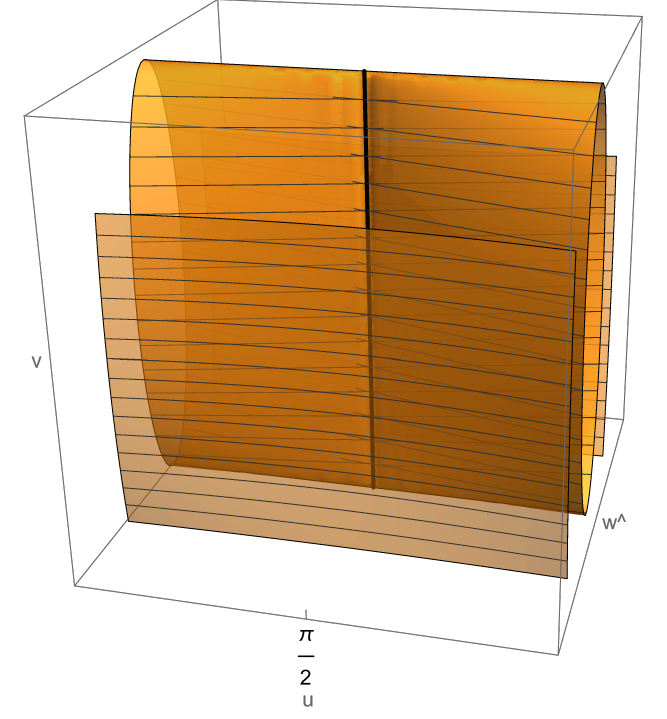}}
\caption{Neighborhood of binary collision in the infinity manifold ($s$=1) and nearby level $s<1$ in $(u,w,v)$ coordinates.  The binary collision at infinity is a line of restpoints but the binary collisions with $s<1$ pass through. }
\label{fig_collisiondetail}
\end{figure}

The proof  involves studying how the arc passes by the restpoints at $\pm\p2$.
Consider a restpoint $q=(s,v,u,w)=(1,v,\fr{\pi}{2},0)$, $-\sqrt{2h}<v<\sqrt{2h}$..
The matrix of the linearization of (\ref{eq_sode}) at $q$ is
\begin{equation}\label{eq_A}
A=\m{0&0&0&0\\
-\fr{1}{\sqrt{2}}&0&0&0&\\
0&0&0&\fr12\\
0&0&\fr12\sqrt{2h-v^2}&0
}.
\end{equation}
The tangent space to the energy manifold at $p$ is $\{\delta s=0\}$.
The eigenvalues of the restriction of $A$ are $0,-\b/2,\b/2$ where $\b=-\sqrt{2h-v^2}$.
In the basis 
$$e_1=(0,1,0,0)\quad  e_2=(0,0,2,-\b)\quad  e_3=(0,0,2,\b)$$
 the matrix  of the restriction of $A$ is
$$\m{0&0&0\\
0&-\b&0\\
0&0&\b
}.
$$
The line segment $q=(1,v,\fr{\pi}{2},0)$,  $-\p2  <v <\p2$  is  a normally  hyperbolic  manifold of equilibrium points.   Each  restpoint has one-dimensional  stable and unstable manifolds both lying in $\cM_\infty$.    They fit together to  give the two-dimensional stable and unstable manifolds of the line segments which is nothing other than the part of the infinity manifold with $-\p2 <v<\p2$, shown in Figure~\ref{fig_collisiondetail}.  The following lemma is the key to the proof of Proposition~\ref{prop_shadowing}.

\begin{proposition}\label{prop_exchange}
Let  $q=(s,v,\t,w)=(1,v_0,\fr{\pi}{2},0)$,  $-\p2  <v_0 <\p2$ be one of the binary collision restpoints and let  $\a(\z)$, $\z\in [0,1)$ be a continuous curve with $\a(0)\in  W^s(q)$ and 
$\a(\z)\in \{s<1\}$ for $\z>0$.   Let  $\Sigma$ be a two-dimensional cross-section to the flow in the energy manifold at a point $q_1\in W^u(q)$.  Then for $\e>0$ sufficiently small, the curve $\a(\z)$, $\z\in (0,\e)$, flows past  $q$ to emerge as a continuous curve $\a_1(\z)$ in $\Sigma$ with $\a(\z)\into q_1$  as $z\into 0$.
\end{proposition}

One could prove this by citing the Hartman-Grobman theorem for normally hyperbolic manifolds \cite{PughShub}.  Namely, the proposition holds for the linear flow determined by (\ref{eq_A}) to which the actual flow is topologically conjugate.  However, as shown by Chen Shi in his thesis on the collinear problem \cite{Shi}, it's possible to give a simple proof independent of the machinery of normally hyperbolic manifolds.  Here is a similar proof adapted to the current situation.
\begin{proof}
Recall that locally near $q$, $\cM_\infty$ is the union of two smooth surfaces (\ref{eq_sheets})  foliated  by the stable and unstable manifolds of $q$ and the nearby restpoints.   It follows that one can introduce smooth coordinates $(x,y,z)$ in the energy manifold making these surfaces the  $(x,z)$ and $(y,z)$ coordinate planes.  The line segment of restpoints is represented in these coordinates by the $z$-axis.  We may also assume the stable and unstable manifolds of these restpoints are given by horizontal lines of constant $z$ within and that the part of the set $\{s<1\}$ containing the curve $\a$ is represented by $\{x>0, y>0\}$.  

In these coordinates, the differential equations will take the form
$$
\begin{aligned}
\dot x &= -\mu x+ xf(x,y,z)\\
\dot y &= \l y + y g(x,y,z)\\
\dot z &= xyh(x,y,z)
\end{aligned}
$$
where $\mu = \l = \b/2$ and where $f,g,h$ are smooth functions with $f(0)=g(0) = 0$.  The proof just uses the property that $\mu>0, \l >0$.  

Choose constants $\mu_1<\mu<\mu_2$ and $\l_1<\l<\l_2$ such that $\l_2-\l_1<\mu_1$.   If $\d>0$ is sufficiently small, we  will have
$$x(t)\le x_0 e^{-\mu_1 t}\qquad  y_0e^{\l_1 t} \le y(t) \le y_0e^{\l_2 t} $$
for $t\ge 0$ as long as the solution remains in $x(t)\in [0,\d], y(t)\in [0,\d], |z(t)|\le \d$.  Furthermore we will have $|h|\le k$ for some constant $k$ in this neighborhood.

We may assume that our curve $\a$ lies in a cross-section to the stable manifold of the form $x = \delta>0$ and is parametrized by continuous functions form $y_0(\z),z_0(\z)$, $\z\in [0,1)$, where $y_0(0)=z_0(0)=0$ and $y_0(\z)>0$ for $\z>0$. We will follow the part of the curve with $y>0$  through the neighborhood to the cross-section $\Sigma=\{y=\d\}$.  The initial conditions are $x_0=\d, y_0(\z)>0, z_0(\z)$.  The time to reach $\Sigma$ satisfies
$$ \fr{1}{\l_2}\ln\fr{\d}{y_0} \le T\le \fr{1}{\l_1}\ln\fr{\d}{y_0}.$$
Meanwhile $\dot z\le k x(t)y(t) \le k \d y_0 e^{(\l_2-\mu_2)t}$ so upon reaching $\Sigma$ we have a curve $(x_1(\z),z_1(\z))$ with
$$|x_1(\z)|\le \d e^{-\mu_1 T} \qquad         |z_1(\z)| \le |z_0(\z)| +k\d y_0(\z)\int_0^T e^{(\l_2-\mu_1)t}\,dt.$$

The curve is continuous for $\z>0$ and we need to show that it has a continuous extension to $\z=0$ with $x_1(0)=z_1(0)=0$.  Since $T\into\infty$ as $\z\into 0$ we have
$x_1(\z)\into 0$ as required.  We need to show  the same for $z_1(\z)$.  Since $z_0(0)=0$ it suffices to consider the term involving the integral.  If $\l_2-\mu_1\le 0$ this term is bounded above by $\fr{k\d}{\l_1}y_0(\z)\ln\fr{\d}{y_0(\z)}\into 0$ as $\z\into 0$.  If $\l_2-\mu_1 >0$ the integral term is 
$$\fr{k\d y_0(\z)}{\l_2-\mu_1}(e^{(\l_2-\mu_1)T}-1)\le \fr{k\d y_0(\z)}{\l_2-\mu_1}\left(\fr{\d}{y_0(\z)}\right)^\fr{\l_2-\mu_1}{\l_1}.$$
This is a constant time $y_0(\z)$ to the power $1-\fr{\l_2-\mu_1}{\l_1}$.  This power is positive by the choice of the constants, so the integral term tends to zero as $\z\into 0$ as required.
\end{proof}

\section{Triple collision to Infinity}\label{sec_collisiontoinfinity}
In this section we prove the existence of connections from the Lagrange equilateral  restpoints on the triple collision manifold to the hyperbolic restpoints at infinity.  It suffices to find forward-time connections between the Lagrange restpoints  with $v>0$ and $\cE_+$.  Then connections from $\cE_-$ to the Lagrange restpoints with $v<0$ follow by  time reversal symmetry.

There are some well-known connections coming from the Lagrange homothetic orbits.   For example there is a solution with constant shape $u=l_+$ and $s$ increasing from $0$ to $1$.  To see this set $u=l_+$ and $w=0$ in equations (\ref{eq_sodesimple}).  The energy equation reduces to $s=\fr{v_+^2-v^2}{v_+^2-1}$ where $v_+=\sqrt{2V(l_+)}$  and the differential equation for $v(\tau)$ becomes 
$$v' = \fr{\cos^2l_+}{2(v_+^2-1)}(v^2-v_+^2)(v^2-1).$$
There is a unique solution $v(\tau)$ to this scalar differential equation connecting the equilibrium points $v=v_+$ and $v=1$ for $-\infty<\tau<\infty$ and the corresponding  $s$ runs from $0$ to $1$.  This gives a homothetic connection from the restpoint $L_+ = (0,v_+,l_+,0)$ in $\cM_0$ to the restpoint $(1,1,l_+,0)\in \cM_\infty$.  Similar homothetic orbits exist connecting $L_-,E \in \cM_0$ to the corresponding point in $\cM_\infty$.

But there  many other connections from $L_+$ to  restpoints $(1,1,u,0)\in \cE_+$.
\begin{proposition}\label{prop_Lplustoinfinity}
Let $p = (1,\sqrt{2h},u,0)\in \cE_+\subset  \cM_\infty$ be any restpoint at infinity with $0< u \le \fr{\pi}{2}$.  Then there is at least one heteroclinic orbit from $L_+$ to $p$.  Similarly, there is at least one heteroclinic orbit from $L_-$ to any restpoint $p$ with $-\fr{\pi}{2}\le u < 0$.  If $m_3<\fr{55}{4}$ there are also infinitely many heteroclinic orbits from each of $L_\pm$ to the Euler restpoint with $u=0$.
\end{proposition}
\begin{proof}
Viewed in the collision manifold $\cM_0$, $L_+$ is a saddle point.  Because $v$ is a Lyapunov function for the flow on the collision manifold, one branch of the unstable manifold 
$W_0^u(L_+)$ converges to the Euler restpoint at $u=0$ while the other spirals around the arm at $u=\fr{\pi}{2}$ (see Figure~\ref{fig_collisionmanifold}). The full unstable manifold $W^u(L_+)$ in the energy manifold has dimension 2.  Consider a semi-circular cross-section to the flow in this unstable manifold parametrized as $q(z), 0\le z\le 1$ where $q(0), q(1)$ are points in the two branches of $W^u_0(L_+)$ and where $q(z)\subset \{s>0\}$ for $0<z<1$.  Now $E$ is a hyperbolic restpoint with a one-dimensional unstable manifold, consisting exactly of the Euler homothetic orbit.  The orbits of $q(z)$ for $z>0$ near $0$ will pass by $E$ and emerge near the unstable manifold.  For the homothetic orbit, $u(\tau)=0$ for all $\tau$ and the limiting shape is $u_+ = 0$.  Since the limiting shape depends continuously on initial conditions we have $u_+(q(z))\into 0$ as $z\into 0^+$.

The orbits of $q(z)$ for $z$ near $1$ will closely follow that branch of $W_0^u(L_+)$ which spirals up the arm of $\cM_0$ near $u=\fr{\pi}{2}$.  The will emerge from a neighborhood 
$\{s\le s_0\}$ with a tight binary configuration and with $v$ arbitrarily large.  From this, it follows that for $|z-1|>0$ sufficiently small,  $u_+(q(z))=\fr{\pi}{2}$.  Here is a proof sketch skipping some details.  Consider an initial condition $(s_0,v_0,u_0,w_0)$ with $s=s_0>0$ and $v_0\ge K$ for some  constant $K>0$.  By choosing $K$  large, the energy equation forces $\cos^2 u_0$ to be arbitrarily small, that is, the shape is a tight binary.  Fixing an upper bounded for $\cos^2 u$, say $\cos^2 u\le \fr12$ makes the quantities $r, r_{13}, |x_2|$ are all comparable in the sense that each is bounded above and below by some positive multiples of the others.  Using these bounds shows that the Jacobi variable $x_2$ satisfies 
$$\ddot x_2 \le -\fr{\a}{x_2^2}$$
for some constant $\a>0$.  On the other hand, by making $v\ge K$ for sufficiently large $K$, we can get an arbitrarily large lower bound for the initial velocity $|\dot x_2(0)|$.  If this initial velocity exceeds the escape velocity for the one-dimensional Kepler problem with force $-\a/x_2^2$, then $|x_2(t)|$ converges monotonically to infinity.  Moreover, we can arrange that the kinetic energy of $x_2(t)$ remains as large as we wish.  As long as this holds, it follows from the Jacobi energy equation that $x_1(t)$ remains bounded and we can arrange that the 
bound $\cos^2 u\le \fr12$ continues to hold.  Therefore the limiting shape is a binary collision which must be $u=\fr{\pi}{2}$.

Since we have $u_+(z)\into 0$ as $z\into 0_+$ and $u_+(z)=\fr{\pi}{2}$ for all  $z$ near $1$,  it follows from the intermediate value theorem that $u_+(q(z))$ take on all values in 
$(0,\fr{\pi}{2}]$.  In other words, there are heteroclinic orbits from $L_+$ to all of the restpoints at infinity with $u$ in this range, as claimed.
When $m_3<\fr{55}{4}$ the eigenvalues at the Euler restpoints $E, E^*$ on the collision manifold $\cM_0$ are nonreal.  As a result, $W_u(L_+)$ forms a spiralling surface around the Euler homothetic orbit.  Following this to a neighborhood of the restpoint $p = (1,\sqrt{2h},0,0)\in \cM_\infty$ shows that $W_u(L_+)$ intersects $W^s(p)$ infinitely often.  

The proofs for $L_-$ follows by symmetry.
\end{proof}

Figure~\ref{fig_lagrangetoinfinity} illustrates Proposition~\ref{prop_Lplustoinfinity} for the case $m_3=1$.   In particular the graph on the right shows some values of the function $u_+(z)$.  Recall that this function is actually analytic near orbits whose limiting shapes are not binary collisions.  It follows from this that values not congruent to $\pm\fr{\pi}{2}$ are only attained finitely many times.  In the figure, it seems that  at least in this case $u_+(z)$ is actually monotonically increasing over much of its domain.

\begin{figure}[h]
\scalebox{0.5}{\includegraphics{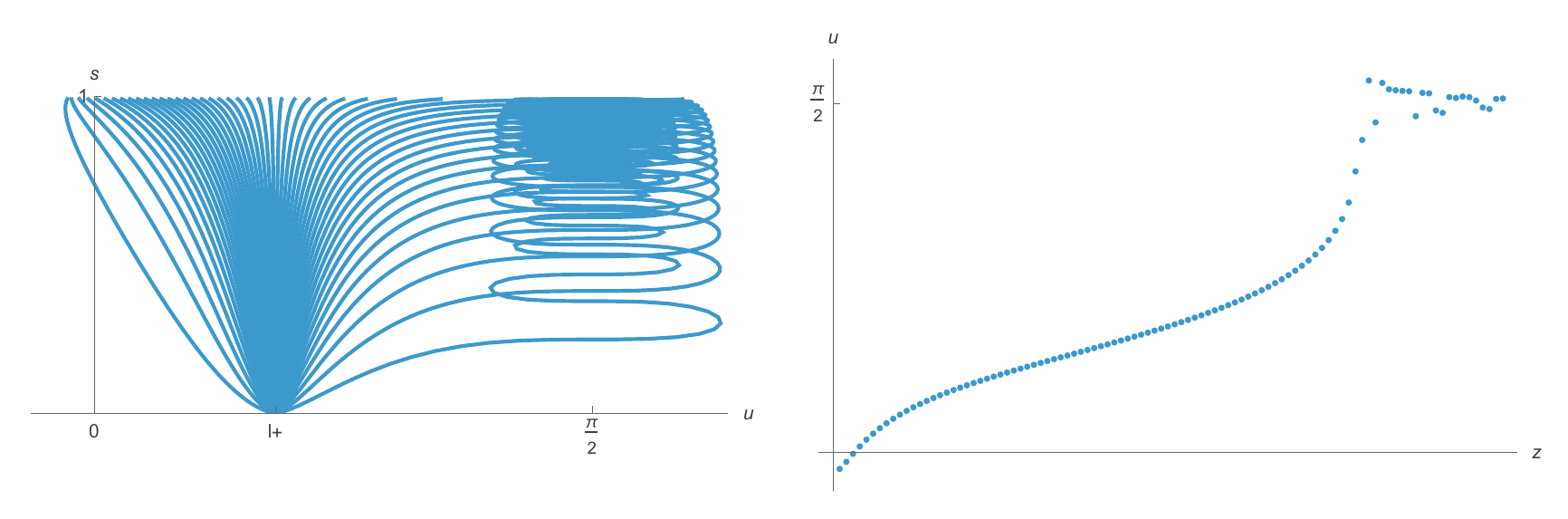}}
\caption{Connections from the restpoint $L_+$ in the collision manifold to various restpoints in the infinity manifold.  The initial conditions are on a small semicircle in $W^u(L+)$ parametrized by $z\in (0,1)$.  On the left, the orbits themselves are shown with $s$ increasing from $0$ to $1$.  The graph on the right shows the final shapes $u_+(z)$.}
\label{fig_lagrangetoinfinity}
\end{figure}

Figure~\ref{fig_lagrangetoinfinity} illustrates Proposition~\ref{prop_Lplustoinfinity} for the case $m_3=1$.   In particular the graph on the right shows some values of the function $u_+(z)$.  Recall that this function is actually analytic near orbits whose limiting shapes are not binary collisions.  It follows from this that values not congruent to $\pm\fr{\pi}{2}$ are only attained finitely many times.  In the figure, it seems that  at least in this case $u_+(z)$ is actually monotonically increasing over much of its domain.

It follows by symmetry and time reversal that there are also connection from the restpoints in $\cE_-$ at  infinity to triple collision.
\begin{cor}\label{cor_infinitytoLplus}
Let $p = (1,-\sqrt{2h},u,0)\in \cE_-\subset \cM_\infty$ be any restpoint at infinity with $0< u \le \fr{\pi}{2}$.  Then there is at least one heteroclinic orbit from  $p$ to $L_+^*$.  Similarly, there is at least one heteroclinic orbit from  any restpoint $p$ with $-\fr{\pi}{2}\le u < 0$ to $L_-^*$.  For $u\ne \pm\fr{\pi}{2}$, there are at most finitely many heteroclinic orbits and, in a cross-section, the intersections of the corresponding stable and unstable manifolds are finite-order crossings of analytic curves.  If $m_3<\fr{55}{4}$ there are also infinitely many heteroclinic orbits from  the Euler restpoint with $u=0$ to each of $L^*_\pm$.

\end{cor}

\section{Scattering theorems}\label{sec_scatteringtheorems}
Now we will put together all of the ingredients from previous sections to prove the main results.  Let $p = (1,-1,u,0)\in \cE_-$ with $0< u < \fr{\pi}{2}$.  Recall that the part of $W^u(p)$ with $s\le 1$ is a two-dimensional half-disc.   To study scattering starting from the shape represented by  $u$, let $\a(z)$, $-1\le z\le 1$, be a parametrization of a semicircular transversal to the flow in the unstable manifolds such that $\a(1)$ is the point of the transversal in  the branch of $W^u_0(p)$ which begins with $w>0$ and $\a(-1)$ is the point of the transversal in branch of $W^u_0(p)$ which begins with $w<0$.  According to Corollary~\ref{cor_infinitytoLplus} there are finitely many parameters $-1<z_i<1$, $i=1,\ldots,n$, such that  $\a(z_i)\in W^s(L_+^*)$.  Let the $z_i$ be listed in increasing order in $(-1,1)$ and define intervals $J(p) =[-1,z_1)$ and $K(p) = (z_n,1]$, that is, the parameter intervals from the endpoints in the collision manifold to the first point which tends to triple collision.  If $m_3<\fr{55}{4}$ there will also be connections from the Euler restpoint with $u=0$.  Although there are infinitely many of these, they will not accumulate near the endpoints $z=\pm1$.  In this case it will still be possible to identify a first and last intersection point and to set up corresponding intervals $J(p), K(p)$.  The following result shows what happens to these intervals in $W^u(p)$ under scattering.  In the statement of the proposition, we allow intervals $(a,b)$ with $b<a$.

\begin{proposition}\label{prop_endpointscattering}
The final shape map $u_+(\a(z))$ is defined and continuous on the intervals $J(p)$ and $K(p)$.  If $m_3\ne \e_1,\e_2$ (the two degenerate values) then  the limits 
$$u_*(p) = \lim_{z\into z_1}u_+(\a(z))\qquad u^*(p) = \lim_{z\into z_n}u_+(\a(z))$$
exist and are equal mod $\pi$ to $0$ or $\p2$.   The initial shape $u$ is related by scattering to all of the shapes with
$$u\in  (u-\pi,u_*)\union (u^*,u+\pi).$$
If $u_*$ or $u^*$ equals $\p2$ mod $\pi$, then those points are also related to $u$ by scattering.   If $m_3<\fr{55}{4}$, the same applies when $u_*$ or $u^*$ equals 0.
\end{proposition} 
\begin{proof}
The points $\a(z)$ with $z\in J(p)\union K(p)$ do not end in triple collision so the scattering map is defined and continuous.  Moreover, Propositions~\ref{prop_totalvariation} and \ref{prop_shadowing} show that
$u_+(\a(\pm1))= u\pm\pi$.  Since the endpoint $z_1$ of $J(p)$ is in $W^s(L_+^*)$, the points  $\a(z)$ with $\d=z_1-z>0$ small will follow one or the other of the branches of $W^u(L_+^*)\subset \cM_0$.  Depending on the value of $m_3$ and the branch followed, Proposition~\ref{prop_collisionwus} shows that the orbits of these points will either approach the Euler restpoint $E$ or spiral up one of the arms of $\cM_0$ (see Figure~\ref{fig_collisionwus}).  In the first case, the orbits emerge near the $W^u(E)$.  Since this is the Euler homothetic orbit, we have  $\lim_{z\into z_1}u_+(\a(z)) = 0\bmod \pi$.   Furthermore, if $m_3<\fr{55}{4}$, the part of the curve with $z\approx z_1$ spirals around the Euler homothetic orbit.  It must cross the stable manifold of the Euler restpoint at infinity and so $0$ is also a scattering value.   In the second case, the orbits with $\d$ sufficiently small emerge from a neighborhood of triple collision with $v$ arbitrarily large.  Then the argument from the proof of Proposition~\ref{prop_Lplustoinfinity} shows that the limiting configuration is the binary collision, that is $\lim_{z\into z_1}u_+(\a(z)) = \p2\bmod \pi$.  In fact all $z$ with $\d$ sufficiently small have $u^+(\a(z))=u_*$ so $u_*$ is also a scattering value.
A similar argument applies to $\lim_{z\into z_n}u_+(\a(z))$.

Since $\a_+$ depends continuously on initial conditions, the rest of the proposition follows from the intermediate value theorem.
\end{proof}

Consider the case $\e_1<m_3<\e_2$ from Proposition~\ref{prop_collisionwus} and note that $\e_2<\fr{55}{4}$.   For the branch of $W^u_0(L^*_+)$ beginning with $w>0$, $u(t)$ increases from $u=l_+$ to $\p2$ and then spirals up the arm of the collision manifold at $u=\fr{3\pi}{2}$.  For the branch beginning with $w<0$, $u(t)$ decreases to $-\p2$ and then spirals up the arm at $u=-\fr{3\pi}{2}$.  Thus the possible values of $u_*, u^*$ are $\pm\fr{3\pi}{2}$.    The leads to cases which we will consider in turn.  If our goal is to find a lot of scattering, the worst case is  $u_*= -\fr{3\pi}{2}$ and $u^*=\fr{3\pi}{2}$.  In this case we have $u$ related by scattering to $[-\fr{3\pi}{2},u-\pi)\union (u+\pi,\fr{3\pi}{2}]$.  If we recall the relation between the shape variables $u,\t$, this translates to the intervals
 $[-\p2,-\t(u))\union (-\t(u),\p2]$ for $\t$ (see Figure~\ref{fig_thetau}).  In other words we can scatter to the whole interval $[-\p2,\p2]$ except possibly $\t(u)$.  Even that value is attained if we include scattering at infinity.
  In the other three case, the whole $\t$-interval $[-\p2,\p2]$ is attained.  For example, if $u_* = \fr{3\pi}{2}$ the  first interval in Proposition~\ref{prop_endpointscattering} would be
 $(u-\pi,\fr{3\pi}{2})$ which clearly covers all shapes. The same argument applies for restpoints $p = (1,-1,u,0)\in \cE_-$ with $\fr{\pi}{2}<u<0$.  So we have:
 
 \begin{proposition}\label{prop_scatteringcaseII}
 If $\e_1<m_3<\e_2$ then each shape $\t_-\ne \pm\p2$ is related by scattering to every shape $\t_+\in [-\p2,\p2]$, except possibly $-\t_-$.
 \end{proposition}

Next suppose $m_3>\e_2$.  One the branch of $W^u_0(L^*_+)$ beginning with $w>0$, $u(t)$ increases from $u=l_+$ to $\p2$, continues past  $u=\fr{3\pi}{2}$ and ends at the Euler restpoint with $u=2\pi$.  For the branch beginning with $w<0$, $u(t)$ decreases to $-\p2$ and then spirals up the arm at $u=-\fr{3\pi}{2}$.  Thus the possible values of $u_*, u^*$ are 
$-\fr{3\pi}{2}$ and $2\pi$.  Once again the possibility with the smallest scattering range is $u_*=  -\fr{3\pi}{2}$ and $u^*=2\pi$.  Then $u$ related by scattering to 
$[-\fr{3\pi}{2},u-\pi)\union (u+\pi,\fr{3\pi}{2}]\union [\fr{3\pi}{2},2\pi]$.  Since $u>0$, this corresponds to the entire $\t$ interval $[-\p2,\p2]$.  Thus
\begin{proposition}\label{prop_scatteringcaseIII}
 If $m_3>\e_2$ then each shape $\t_-\ne 0,\pm\p2$ is related by scattering to every shape $\t_+\in [-\p2,\p2]$.  If $m_3<\fr{55}{4}$ this is also true for the Euler shape $\t_-=0$.
 \end{proposition}
 
 Finally, consider the case $m_3<\e_1$.  This time  the branch of $W^u_0(L^*_+)$ beginning with $w>0$ has $u(t)$ increasing from $u=l_+$ to  $u=\fr{3\pi}{2}$ while  the branch beginning with $w<0$ decreases to the Euler configuration  $u=-\pi$.  The possible values of $u_*, u^*$ are 
$-\pi$ and $\fr{3\pi}{2}$.  If these are assigned to $u_*$ ands $u^*$, respectively, then the scattering intervals are 
$[-\pi,u-\pi)\union (u+\pi,\fr{3\pi}{2}]$ and the corresponding $\t$ values cover $[-\p2,-\t)\union (-\t,0]$ which amounts only to about half of the full range.   Other possible assignments of $u_*, u^*$ lead to more scattering.

\begin{proposition}\label{prop_scatteringcaseI}
 If $m_3<\e_1$ then each shape $\t_-\in [0,\p2)$ is related by scattering to every shape $\t_+\in [-\p2,0]$, except possibly $-\t_-$.  Each shape shape $\t_-\in (-\p2,0]$ is related by scattering to every shape $\t_+\in [0,\p2]$, except possibly $-\t_-$.
 \end{proposition}
 
 Propositions~\ref{prop_scatteringcaseI}, \ref{prop_scatteringcaseII} and \ref{prop_scatteringcaseIII} summarize the effect of the binary and triple collisions on scattering.  The binary collision is responsible for the scattering from $\t$ to $-\t$ at infinity while the triple collision determines the values of $u_*,u^*$.  Once these endpoints are determined, continuity of the scattering map gives scattering to  the nontrivial intervals in propositions.  Of course there may be additional scattering beyond that which can be proved by this method.   This will be evident from the experiments described in the next section.
 
 \section{Scattering experiments}\label{sec_scatteringexperiments}
 To illustrate Proposition~\ref{prop_scatteringcaseII}, consider the equal mass case $m_3=1$ with energy $h=\fr12$.  For our initial shape we will choose the Lagrange equilateral triangle $u_-=l_+$.  Let $p=(s,v,u,w)=(1,-1,l_+,0)\in\cE_-$ be the corresponding restpoint at infinity.  We choose several initial conditions in $W_u(p)$ and follow them numerically until they return to the infinity manifold near some point $q=(1,1,u_+,0)\in\cE_+$.  Figure~\ref{fig_scatterm31thetalag} shows these orbits in the $(\t,v)$ plane.  According to Proposition~\ref{prop_scatteringcaseII}, all values of $\t_+$ except perhaps $-\t_-$ are attained.  
 
 Figure~\ref{fig_thetaplusgraphm31} shows a graph of the final angle $\t_+$ as function of the parameter $z$ describing a semicircular crosssection in $W_u(p)$ where the endpoints are in $\cM_\infty$.  The discontinuity occurs at the parameter value of an orbit connecting $p$ to triple collision.  The result is consistent with Propositions~\ref{prop_endpointscattering} and \ref{prop_scatteringcaseII}.

According to Proposition~\ref{prop_scatteringcaseII}, similar results should be obtained starting from any initial shape.  Figures~\ref{fig_thetaplusgraphm31alt} and \ref{fig_thetaplusgraphm31euler} shows the scattering graph starting from a different hyperbolic restpoint $p=(s,v,u,w)=(1,-1,u_-,0)\in\cE_-$.  This time the there are several orbits in $W^u(p)$ which end in triple collision.  The behavior near the endpoints is consistent with Propositions~\ref{prop_endpointscattering} and \ref{prop_scatteringcaseII}.

\begin{figure}[h]
\scalebox{0.4}{\includegraphics{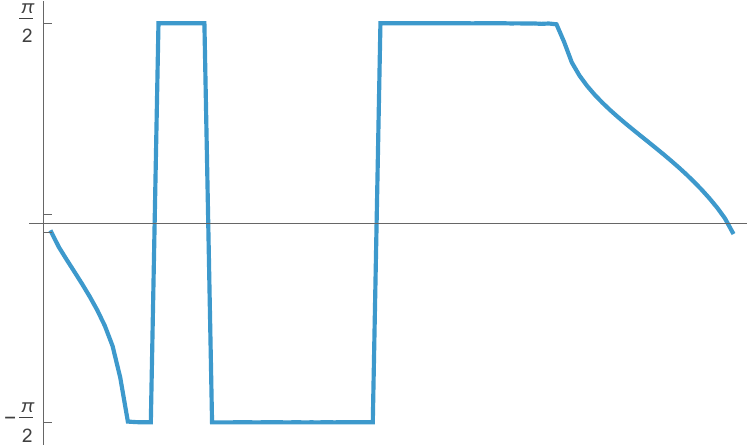}}
\caption{Scattering from another initial shape $\t_-\approx 0.225$ in the equal mass isosceles three-body problem.  This time there are several triple collisions.}
\label{fig_thetaplusgraphm31alt}
\end{figure}

 \begin{figure}[h]
\scalebox{0.4}{\includegraphics{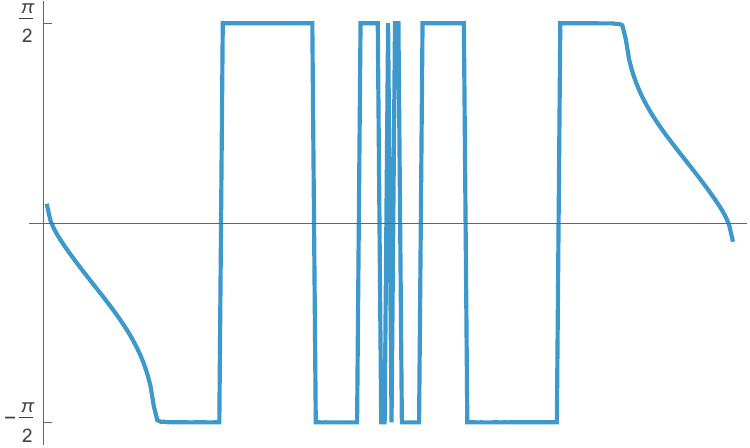}}
\caption{Scattering from the Euler shape $\t_-=0$ in the equal mass isosceles three-body problem hinting at infinitely many triple collisions.}
\label{fig_thetaplusgraphm31euler}
\end{figure}

 To illustrate Proposition~\ref{prop_scatteringcaseIII}, we used $m_3=3$.    The scattering starting from the Lagrange shape is illustrated in Figure~\ref{fig_scatterinfom33}.  This time the whole interval $\t_+\in [-\p2,\p2]$ is covered.  The analogous results for  $m_3=0.3$ are shown in Figure~\ref{fig_scatterinfom303}.  This time the values final scattering angle is confined to a proper subinterval of $[-\p2,\p2]$ which, in accord with Proposition~\ref{prop_scatteringcaseI} contains at least $[-\p2,0]\setminus\{-\t_-\}$.

\begin{figure}[h]
\scalebox{0.4}{\includegraphics{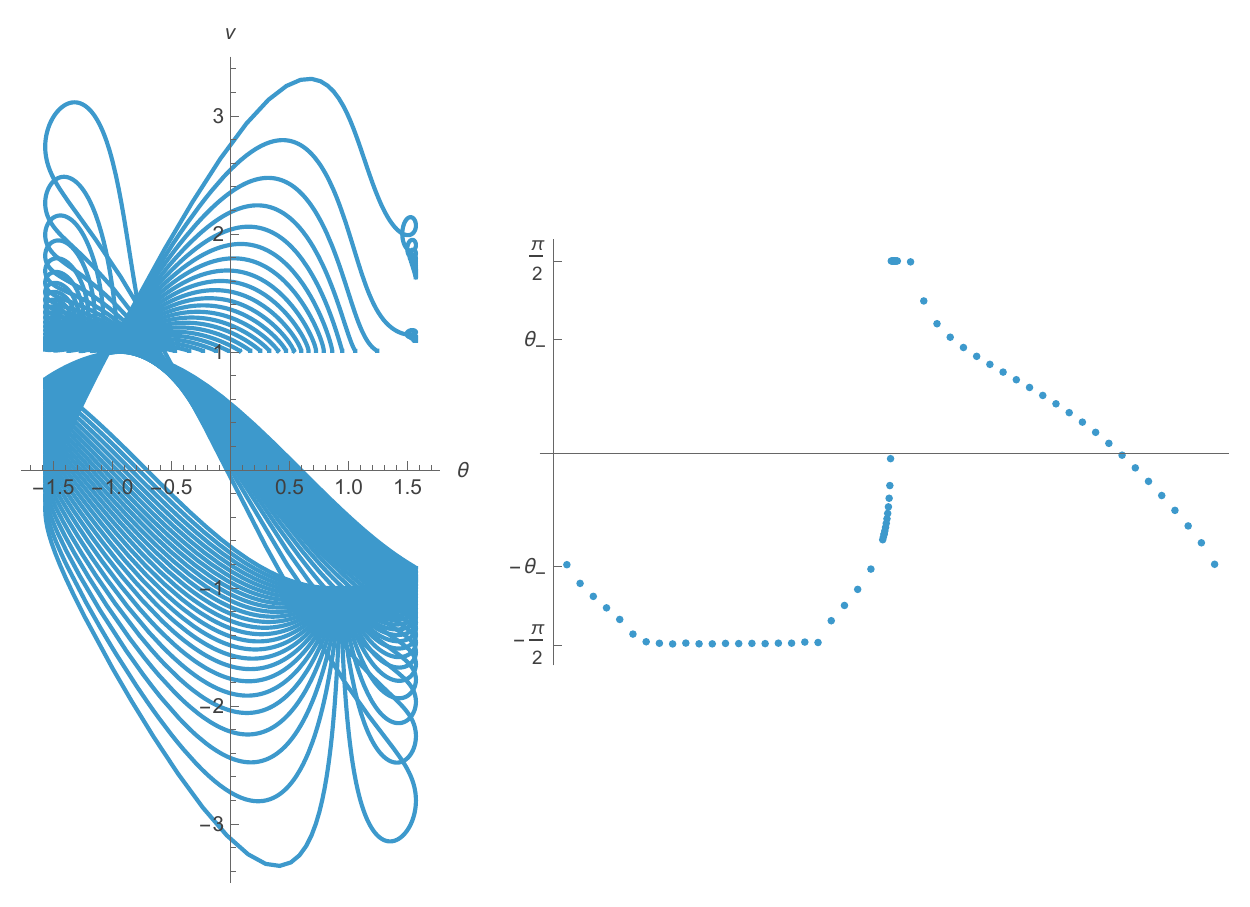}}
\caption{Scattering from the Lagrange equilateral triangle for $m_3=3>\e_2$. All values of $\t_+\in [-\p2,\p2]$ are achieved.}
\label{fig_scatterinfom33}
\end{figure}

\begin{figure}[h]
\scalebox{0.4}{\includegraphics{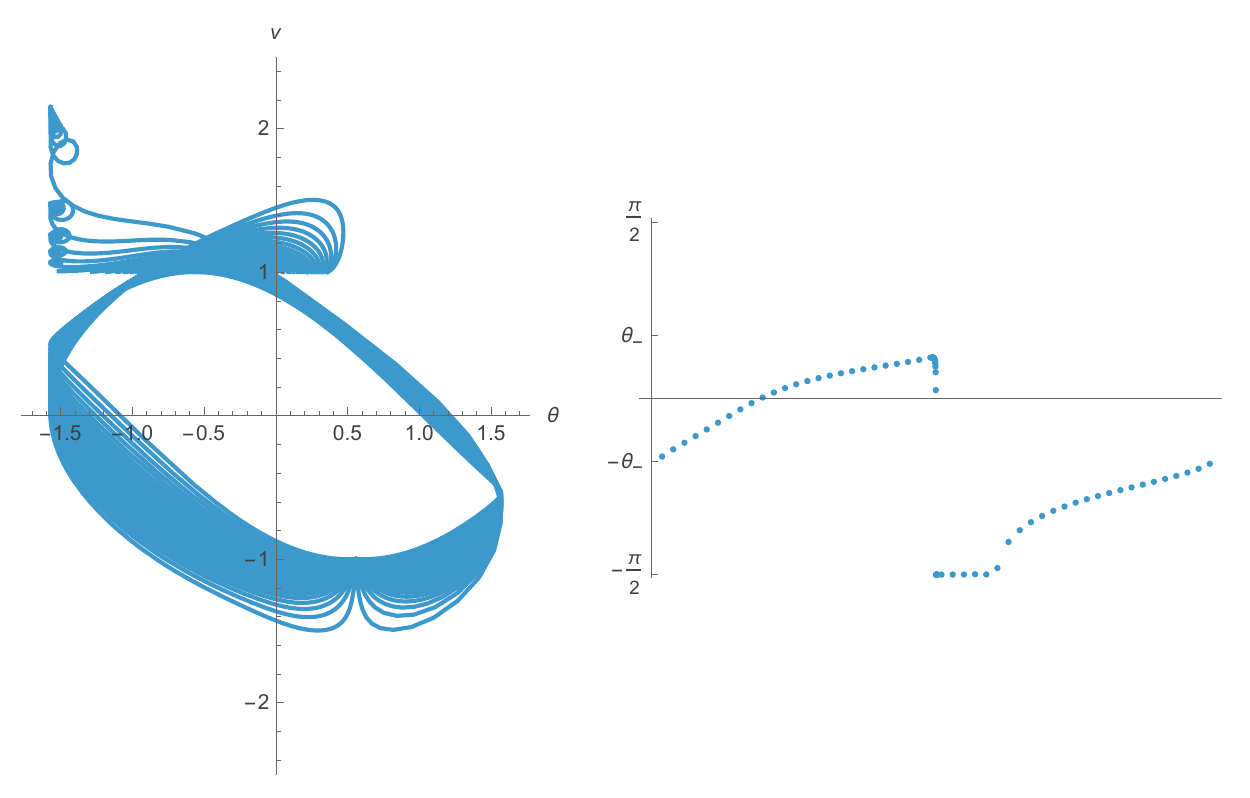}}
\caption{Scattering from the Lagrange equilateral triangle for $m_3=3>\e_2$.  $\t_+$ covers at least $[-\p2,0]\setminus \{-\t_-\}$.}
\label{fig_scatterinfom303}
\end{figure}

\vfill
\vfill
\break

\end{document}